\documentclass[a4paper]{amsart}
\usepackage{mathmarco}

\usepackage[maxalphanames=99,url=false,maxbibnames=99,giveninits=true]{biblatex}
\AtEveryBibitem{
  \ifentrytype{article}{\clearfield{doi}\clearfield{issn}\clearfield{isbn}}{}
  \ifentrytype{incollection}{\clearfield{doi}\clearfield{issn}\clearfield{isbn}}{}
  \ifentrytype{book}{\clearfield{doi}\clearfield{issn}\clearfield{isbn}}{}
}

\usepackage{etoolbox}

\makeatletter

\@ifpackageloaded{hyperref}{
  \@ifpackageloaded{biblatex}{
    \DeclareRobustCommand{\SkipTocEntry}[5]{}
  }
}{}

\let\origsubsection\subsection
\renewcommand{\subsection}{
  \@ifstar{\subsection@star}{\origsubsection}%
}
\newcommand{\subsection@star}[1]{
  \addtocontents{toc}{\protect\SkipTocEntry}
  \origsubsection*{#1}
}

\let\origsection\section
\renewcommand{\section}{
  \@ifstar{\section@star}{\origsection}%
}
\newcommand{\section@star}[1]{
  \addtocontents{toc}{\protect\SkipTocEntry}
  \origsection*{#1}
}

\newcommand{\addresseshere}{
  \enddoc@text\let\enddoc@text\relax
}

\renewcommand*\env@matrix[1][\arraystretch]{%
  \edef\arraystretch{#1}%
  \hskip -\arraycolsep
  \let\@ifnextchar\new@ifnextchar
  \array{*\c@MaxMatrixCols c}}

\makeatother

\title[Optimal Hypercontractivity and Log--Sobolev inequalities on Cyclic Groups $\mathbb{Z}_{n}$]{A Complete Solution of Optimal Hypercontractivity and Log--Sobolev inequalities on Cyclic Groups $\mathbb{Z}_{n}$ for $n\geq 4$}

\date{\today}

\author{Gan Yao}

\address[Gan Yao]{Institute for Advanced Study in Mathematics, Harbin Institute of Technology,  Harbin 150001, China.}
\email{gan.yao.math@gmail.com}
\begin{document}

\begin{abstract}
For $1<p\le q<\infty$ and $n\geq 4$, we prove that the Poisson-like semigroup $(P_t)_{t\in \mathbb{R}_+}$ on $\mathbb{Z}_n$, associated with the word length $\psi_n(k)=\min(k,n-k)$, is hypercontractive from $L_p$ to $L_q$ if and only if $t\ge \tfrac{1}{2}\log\big(\tfrac{q-1}{p-1}\big)$. To this end, we establish the corresponding sharp Log--Sobolev inequalities with the optimal constant $2$.
\end{abstract}
\maketitle
\tableofcontents
\section{Introduction}
Let $\mathbb{Z}_n=\{0,1,\dots,n-1\}$ denote the finite cyclic group equipped with the normalized counting measure $\mu_n$. The Poisson-like semigroup $(P_t)_{t\in \mathbb{R}_+}$ applied to $f(x)=\sum_{k=0}^{n-1} a_k\chi_k(x)\in L_\infty(\mathbb{Z}_n,\mu_n)$ is given by
\begin{equation}\label{Eqn: Formula of P_t}
  (P_tf)(x)=\sum_{k=0}^{n-1} e^{-t\psi_n(k)}a_k\chi_k(x),
\end{equation}
where $\psi_n(k)=\min(k,n-k)$ is the word-length function on $\mathbb{Z}_n$, and $\chi_k(x)=e^{\frac{2\pi ikx}{n}}\in L_\infty(\mathbb{Z}_n)$. The hypercontractivity problem is to determine the optimal time $t_{p,q}$ with $1< p\leq q<\infty$ such that
\[
    \norm{P_t f}_{q}\le \norm{f}_{p}\qquad \text{for all } t\ge t_{p,q}.
\]
A standard series expansion of $\norm{P_t(1+\epsilon f)}_q$ and $\norm{1+\epsilon f}_p$ at $\epsilon=0$ shows the universal lower bound 
\[
  t_{p,q}\ge \tfrac{1}{2}\log\big(\tfrac{q-1}{p-1}\big).
\]\par 
The main result of this article is that this bound is sharp for every $n\ge 4$.
\begin{theorem}\label{Thm: Hypercontractivity on n}
  For $n\geq 4$, we have 
  \[
    \norm{P_{t}f}_q\le \norm{f}_p\quad \Leftrightarrow \quad t\ge \frac{1}{2}\log(\frac{q-1}{p-1})
    \]
    for $1<p\le q<\infty$.
\end{theorem}
Before the present work, the known results were confined to a few special cases. For the case $n=2$, the optimal time $t_{p,q}=\tfrac{1}{2}\log\big(\tfrac{q-1}{p-1}\big)$ follows from the classical two-point inequality. That inequality was first proved by Bonami~\cite{MR283496}, and later rediscovered by Nelson~\cite{MR343816}, Gross~\cite{MR420249}, and Beckner~\cite{MR385456}. The case $n=3$ is drastically different: using Wolff's reduction~\cite[Corollary 3.1]{MR2314078}, it was established in \cite{MR1796718,MR2073432} that the optimal time $t_{2,q}$ for $P_t$ on $\mathbb{Z}_3$ is
\[
t_{2,q}=\frac{1}{2}\log\!\Bigg(\frac{\frac{2}{3}\big(\frac{1}{3}\big)^{\frac{2}{q}-1}-\frac{1}{3}\big(\frac{2}{3}\big)^{\frac{2}{q}-1}}{\big(\frac{2}{3}\big)^{\frac{2}{q}}-\big(\frac{1}{3}\big)^{\frac{2}{q}}}\Bigg),
\]
whereas the optimal time for general $p,q$ was determined recently by Cao, Fan, Han, Qiu, and Wang~\cite{cao2026optimalhypercontractiveconstantsmathbbz3}. For $n\geq 4$, a few isolated cases were studied:
\begin{itemize}
  \item $n=4$: $t_{p,q}=\tfrac{1}{2}\log\big(\tfrac{q-1}{p-1}\big)$ by Beckner, Janson, and Jerison~\cite{MR730056} in 1983. They also conjectured that this optimal time holds for all $n\neq 3$.
  \item $n=5$: $t_{2,q}=\tfrac{1}{2}\log(q-1)$ for $q\in 2\mathbb{Z}_+$ by Andersson~\cite{MR1883499} in 2002.
  \item $n\geq 6$: $t_{2,q}=\tfrac{1}{2}\log(q-1)$ for $q\in 2\mathbb{Z}_+$ for even $n$ and the same holds for odd $n$ whenever $n\geq q$ by Junge, Palazuelos, Parcet, and Perrin~\cite{MR3709719} in 2017.
  \item $n\in \{2^k,3\cdot 2^k\}$ with $k\geq 1$: $t_{p,q}=\tfrac{1}{2}\log\big(\tfrac{q-1}{p-1}\big)$ by our work \cite{yao2025optimalhypercontractivitylogsobolevinequalities} in 2025.
\end{itemize} \par 
Hypercontractivity has also been extensively studied for semigroups on manifolds. Weissler~\cite{MR578933} proved optimal hypercontractivity for the heat and Poisson semigroups on the circle $\mathbb{S}^1$, while Rothaus~\cite{MR593787} independently treated the heat case. For related results on the sphere $\mathbb{S}^n$, see \cite{MR674060,MR706641,MR1164616,MR4278128}.\par 

It is known that the word-length function $\psi_n$ is conditionally negative; see, e.g. \cite{MR3709719}. In this setting, a standard route to hypercontractivity proceeds through Log--Sobolev inequalities (LSI), following Gross's celebrated work~\cite{MR420249}: $\norm{P_{t}f}_q\le \norm{f}_p$ holds whenever $t\geq \frac{C}{4}\log\big(\frac{q-1}{p-1}\big)$ if and only if the corresponding LSI holds with constant $C$. Thus it is enough to establish the following $n$--LSI with the optimal constant $2$ for all $n\geq 4$. The case $n=4$ follows from \cite{MR730056} together with Gross's extrapolation technique~\cite{MR420249}, while $n=3$ is rather different; see \cite{MR1758914,MR1410112}.
\begin{theorem}\label{Thm: Log Sobolev inequality n>=4}
For $n\geq 4$, we have the following LSI with the optimal constant $2$:
\[
  \int_{\mathbb{Z}_n} f^2\log f^2\dd\mu_n-\norm{f}_2^2\log\norm{f}_2^2\le 2\inner{f,A_{\psi_n}f}_{L_2(\mathbb{Z}_n,\mu_n)},\qquad f\in L_2^+(\mathbb{Z}_n,\mu_n),
\]
where $A_{\psi_n}$ is the generator of the semigroup $(P_t)_{t\in\mathbb{R}_+}$. 
\end{theorem}
As a consequence, we recover Weissler's hypercontractive estimate and LSI for the Poisson semigroup on the circle $\mathbb{S}^1$. The Poisson semigroup $(\mathcal{P}_t)_{t\in \mathbb{R}_+}$ applied to $f(\theta)=\sum_{m=-\infty}^{\infty} a_me^{im\theta}\in L_2(\mathbb{S}^1,\frac{\dd \theta}{2\pi})$ is given by 
  \[
      (\mathcal{P}_tf)(\theta)=\sum_{m=-\infty}^{\infty} a_me^{-t\abs{m}}e^{im\theta}.
  \]
\begin{corollary}[{\cite{MR578933}}]\label{Coro: LSI on circle}
  
  We have 
  \[
    \norm{\mathcal{P}_{t}f}_{L^q(\mathbb{S}^1)}\le \norm{f}_{L^p(\mathbb{S}^1)}\quad \Leftrightarrow \quad t\ge \frac{1}{2}\log(\frac{q-1}{p-1})
    \]
    for $1<p\le q<\infty$. Equivalently, for every $f\geq 0$ almost everywhere with Fourier expansion $f(\theta)=\sum_{m=-\infty}^{\infty} a_me^{im\theta}$, we have
  \[
    \int f^2\log f\leq \sum_{m=-\infty}^{\infty} \abs{m}\abs{a_m}^2+\norm{f}_{L^2(\mathbb{S}^1)}^2\log\norm{f}_{L^2(\mathbb{S}^1)}
  \] 
  provided the right-hand side is finite.
\end{corollary}

It is worth noting that Theorem~\ref{Thm: Log Sobolev inequality n>=4} also yields the sharp LSI for the simple random walk on $\mathbb{Z}_n$. Earlier exact results were known for all even $n$ and for odd $n\le 21$; see \cite{MR1410112,MR1990534,MR2487855,MR4703690}.\par 
\begin{corollary}\label{Coro: LSI for simple random walk}
Let $K$ be the Markov kernel of the simple random walk on $\mathbb{Z}_n$, given by
\[
K(x,y)=\begin{cases}
\frac{1}{2}, & y=x\pm 1,\\
0, & \text{otherwise},
\end{cases}
\]
with indices taken modulo $n$. Then, for every $n\ge 4$, we have the following LSI with the optimal constant $\frac{2}{1-\cos(\frac{2\pi}{n})}$:
\begin{equation}\label{Ineq: LSI for simple random walk}
   \int_{\mathbb{Z}_n} f^2\log f^2\dd\mu_n-\norm{f}_2^2\log\norm{f}_2^2\le \frac{2}{1-\cos(\frac{2\pi}{n})}\inner{f,(I-K)f}_{L_2(\mathbb{Z}_n,\mu_n)},\qquad f\in L_2^+(\mathbb{Z}_n,\mu_n).  
\end{equation}
\end{corollary}

The main novelty of this work is a reduction that converts an optimization problem involving \textbf{logarithmic terms} into the analysis of explicit \textbf{low-degree polynomials}. More precisely, we show that the $n$--LSI follows from the nonnegativity of an explicit polynomial of degree at most $3$ on the region $\{\lambda\in\mathbb{R}_+^n:0<\norm{\lambda}_2^2<n\}$. We further observe that this polynomial admits a simple representation on the Fourier side, which enables suitable estimates. Building on this observation, we carry out a substantial technical analysis of the resulting polynomial inequalities, based on coefficient decompositions with careful use of arithmetic--geometric mean arguments.
\medskip

The article is organized as follows. For completeness, Section~\ref{Sec: Preliminary results and definitions} reviews the Karush--Kuhn--Tucker (KKT) approach to the $n$--LSI from \cite{yao2025optimalhypercontractivitylogsobolevinequalities}. In Section~\ref{Sec: Reduction to a polynomial inequality}, we reduce the $n$--LSI to a polynomial inequality 
\[
Q_{n,2}(x)+Q_{n,4}(x)-Q_{n,3}(x)\ge 0,\qquad x\in \mathbb{R}_+^{\floor{n/2}},
\]
where each $Q_{n,d}$ is a homogeneous polynomial of degree $d$. Section~\ref{Sec: Proof of key lemma, odd n} proves this inequality in the odd case, with $n=5$ treated separately, and Section~\ref{Sec: Proof of key lemma, even n} deals with the even case. 

\section{Preliminary results and definitions}\label{Sec: Preliminary results and definitions}

Let $F_n$ denote the $n\times n$ discrete Fourier transform (DFT) matrix
\[
    F_n=\begin{pmatrix}
        1 & 1 & 1 & \cdots & 1\\
        1 & \omega & \omega^2 & \cdots & \omega^{n-1}\\
        1 & \omega^2 & \omega^4 & \cdots & \omega^{2(n-1)}\\
        \vdots & \vdots & \vdots & \ddots & \vdots\\
        1 & \omega^{n-1} & \omega^{2(n-1)} & \cdots & \omega^{(n-1)(n-1)}
    \end{pmatrix},
\]
where $\omega=e^{\frac{2\pi i}{n}}$ and $(F_n)_{j,k}=\omega^{jk}$ for $0\le j,k\le n-1$. It is well known that $\frac{1}{\sqrt{n}}F_n$ is unitary (see, e.g., \cite[Section 2.5]{MR543191}). For a column vector $x\in\mathbb{C}^n$, we write its DFT as
\[
      \hat{x}=(\hat{x}_0,\ldots,\hat{x}_{n-1})^{\mathrm{T}}=F_nx,\qquad \hat{x}_k=\sum_{j=0}^{n-1}x_j\omega^{jk},\quad 0\le k\le n-1.
\]\par 
Given a non-zero vector $\lambda\in\mathbb{R}_+^n$, define the entropy functional on the vector $\lambda$ by
\[
  \mathrm{H}_n[\lambda] \coloneqq\frac{1}{n}\sum_{k=0}^{n-1}\lambda_k^2\log(\lambda_k^2)-\frac{\norm{\lambda}_2^2}{n}\log(\frac{\norm{\lambda}_2^2}{n})=\frac{1}{n}\sum_{k=0}^{n-1}\lambda_k^2\log(\frac{n\lambda_k^2}{\norm{\lambda}_2^2}),
\]
with the convention $0\log 0=0$. Set
\[
  \diag(\psi_n)\coloneqq \diag\big(\psi_n(0),\psi_n(1),\ldots,\psi_n(n-1)\big)\in M_n(\mathbb{R}),
\]
where $\psi_n(k)=\min(k,n-k)$ is the word-length function on $\mathbb{Z}_n$. Let 
\[
    \Psi(n)=\frac{1}{n}F_n\diag(\psi_n)F_n^{-1}\in M_n(\mathbb{C}).
\]
By identifying $f$ with its value vector $(f(0),\dots,f(n-1))^{\mathrm{T}}$ in $\mathbb R^n$, an easy consequence of the unitarity of $\frac{1}{\sqrt{n}}F_n$ is that the $n$--LSI 
\[
  \int_{\mathbb{Z}_n} f^2\log f^2\dd\mu_n-\norm{f}_2^2\log\norm{f}_2^2\le 2\inner{f,A_{\psi_n}f}_{L_2(\mathbb{Z}_n,\mu_n)},\qquad f\in L_2^+(\mathbb{Z}_n,\mu_n),
\]
is equivalent to
\begin{equation}\label{Eqn: n-LSI with H_n(lambda)}
   \mathrm{H}_n[\lambda] \le 2\inner{\lambda,\Psi(n)\lambda},\qquad \lambda\in \mathbb{R}_+^n.
\end{equation}
Moreover, $\Psi(n)$ is a real symmetric matrix (see, e.g., \cite[Lemma 2.1]{yao2025optimalhypercontractivitylogsobolevinequalities}). \par 

Since both $\inner{\lambda,\Psi(n)\lambda}$ and $\mathrm{H}_n[\lambda]$ are homogeneous of degree $2$, it suffices to study the functional $g(\lambda)\coloneqq 2\inner{\lambda,\Psi(n)\lambda}-\mathrm{H}_n[\lambda]$ on the nonnegative part of the unit sphere 
\[
    \mathbb{S}^{n-1}_+ \coloneqq \Big\{ \lambda\in\mathbb{R}_+^n : \sum_{k=0}^{n-1}\lambda_k^2=1 \Big\}.
\]
Accordingly, we consider the optimization problem
\[
    \min_{\lambda\in\mathbb{S}^{n-1}_+}2\inner{\lambda,\Psi(n)\lambda}-\mathrm{H}_n[\lambda],
\]
and analyze the associated Karush--Kuhn--Tucker (KKT) system. In particular, to prove the nonnegativity of $2\inner{\lambda,\Psi(n)\lambda}-\mathrm{H}_n[\lambda]$, it suffices to show that a modification of the corresponding KKT system has no solution~\cite[Lemma 2.3]{yao2025optimalhypercontractivitylogsobolevinequalities}. We record this criterion from \cite{yao2025optimalhypercontractivitylogsobolevinequalities} for later use.
\begin{lemma}[{\cite[Lemma 2.3]{yao2025optimalhypercontractivitylogsobolevinequalities}}]\label{Lem: KKT condition absorbing mu}
  Let $Q\in M_n(\mathbb{R})$ be a symmetric matrix and set
\[
    g(\lambda)=2\inner{\lambda,Q\lambda}-\mathrm{H}_n[\lambda].
\]
If the system
\begin{equation}\label{Eqns: KKT condition absorbing mu}
\begin{cases}
4Q\lambda-\dfrac{4}{n}\begin{pmatrix}
\lambda_0\log(\lambda_0)\\
\vdots\\
\lambda_{n-1}\log(\lambda_{n-1})
\end{pmatrix}-\nu=0,\\
0<\norm{\lambda}_2^2<n,\\
\lambda_j\geq 0,\qquad 0\le j\le n-1,\\
\lambda_j\nu_j =0,\qquad 0\le j\le n-1,\\
\nu_j\ge 0,\qquad 0\le j\le n-1,
\end{cases}
\end{equation}
has no solution, then $g\ge 0$ on $\mathbb{S}^{n-1}_+$.
\end{lemma}

\section{Reduction to a polynomial inequality}\label{Sec: Reduction to a polynomial inequality}
The key novelty of our work is to reduce the existence of solutions to \eqref{Eqns: KKT condition absorbing mu} to an explicit polynomial inequality of degree at most three, since the original system \eqref{Eqns: KKT condition absorbing mu} is typically intractable.\par 
Let $S_n$ be the $n\times n$ cyclic shift matrix, so that
\[
 (S_n\lambda)_i=\lambda_{i+1},
\]
where all indices are taken modulo $n$. Define the forward difference operator on the $n$-cycle by
\[
\nabla \lambda\coloneqq (I-S_n)\lambda,\qquad (\nabla\lambda)_i=\lambda_i-\lambda_{i+1},
\]
and the discrete Laplacian $L=\nabla^*\nabla$ by
\[
L\lambda\coloneqq (I-S_n^\mathrm{T})(I-S_n)\lambda=(2I-S_n-S_n^{\mathrm{T}})\lambda,\qquad (L\lambda)_i=2\lambda_i-\lambda_{i+1}-\lambda_{i-1}.
\]
For $\lambda=(\lambda_0,\dots,\lambda_{n-1})^\mathrm{T}\in\mathbb{R}^n$, write
\[
\lambda^2 \coloneqq (\lambda_0^2,\dots,\lambda_{n-1}^2)^\mathrm{T}.
\]
We define the functional $K_n:\mathbb{R}_+^n\to\mathbb{R}$ by
\begin{equation}\label{Def: Functional Kn}
K_n(\lambda)\coloneqq
4\inner{\lambda,L\Psi(n)\lambda}
-\frac{2}{n}\inner{\lambda^2,L\lambda}.
\end{equation}
\begin{theorem}\label{Thm: nonnegativity of K_n implies n--LSI}
  Assume that for every $\lambda\in \mathbb{R}_+^n$ satisfying $0<\norm{\lambda}_2^2<n$, one has
\begin{equation}\label{Eqn: Kn nonnegative}
K_n(\lambda)\ge 0.
\end{equation}
Then the system \eqref{Eqns: KKT condition absorbing mu} admits no solution. Consequently, the $n$--LSI \eqref{Eqn: n-LSI with H_n(lambda)} holds.
\end{theorem}
\begin{proof}
  Suppose, for contradiction, that $(\lambda^*,\nu^*)$ is a solution of \eqref{Eqns: KKT condition absorbing mu}. We first show that $\lambda^*$ cannot be a constant vector. Indeed, if $\lambda^*=c(1,\dots,1)=c\mathbf{1}$, then $0<\norm{\lambda^*}_2^2=nc^2<n$ implies $c\in(0,1)$. Since every coordinate of $\lambda^*$ is strictly positive, the condition $\lambda_j^*\nu_j^*=0$ and $\nu^*_j\geq 0$ for all $j$ gives $\nu^*=0$. On the other hand, $\mathbf{1}$ is an eigenvector of $\Psi(n)$ with eigenvalue $0$, so the first equation in \eqref{Eqns: KKT condition absorbing mu} becomes
\[
-\frac{4}{n}c\log c\,\mathbf{1}=\mathbf{0},
\]
which is impossible because $c\log c\neq 0$ for $0<c<1$.\par 

Next we claim that
\begin{equation}\label{Eqn: secant estimate}
\frac{a\log a-b\log b}{a-b}<\frac{a+b}{2},
\qquad a,b\geq 0,\quad a\neq b,
\end{equation}
with the convention $0\log 0=0$. If $a,b>0$, then the Hermite--Hadamard inequality
\[
   \frac{1}{b-a}\int_a^b f(x)\dd x<f(\frac{a+b}{2})
\]
applied to the strictly concave function $x\mapsto 1+\log x$ on $[\min(a,b),\max(a,b)]$ gives
\[
\frac{a\log a-b\log b}{a-b} < 1+\log\left(\frac{a+b}{2}\right)\leq \frac{a+b}{2}.
\]
If $b=0$, then
\[
\frac{a\log a-b\log b}{a-b}=\log a<\frac{a}{2}=\frac{a+b}{2},
\]
so \eqref{Eqn: secant estimate} holds in all cases.

For each $i$, define
\[
T_{n,i}(\lambda)\coloneqq \bigl(4(\Psi(n)\lambda)_i-4(\Psi(n)\lambda)_{i+1}\bigr)(\lambda_i-\lambda_{i+1})-\frac{2}{n}(\lambda_i+\lambda_{i+1})(\lambda_i-\lambda_{i+1})^2.
\]
If $\lambda_i^*\neq \lambda_{i+1}^*$, then by the first equation in
\eqref{Eqns: KKT condition absorbing mu} and \eqref{Eqn: secant estimate}, we have
\[
\frac{4(\Psi(n)\lambda^*)_i-4(\Psi(n)\lambda^*)_{i+1}-(\nu_i^*-\nu_{i+1}^*)}{\lambda_i^*-\lambda_{i+1}^*}=\frac{4}{n}\frac{\lambda_i^*\log\lambda_i^*-\lambda_{i+1}^*\log\lambda_{i+1}^*}{\lambda_i^*-\lambda_{i+1}^*}<\frac{2}{n}(\lambda_i^*+\lambda_{i+1}^*).
\]
Multiplying both sides by $(\lambda_i^*-\lambda_{i+1}^*)^2>0$, we obtain
\[
T_{n,i}(\lambda^*) < (\nu_i^*-\nu_{i+1}^*)(\lambda_i^*-\lambda_{i+1}^*).
\]
Moreover, since $\lambda_j^*,\nu_j^*\geq 0$ and $\lambda_j^*\nu_j^*=0$ for all $j$, we have
\[
(\nu_i^*-\nu_{i+1}^*)(\lambda_i^*-\lambda_{i+1}^*)=-\nu_i^*\lambda_{i+1}^*-\nu_{i+1}^*\lambda_i^*\leq 0.
\]
Hence $T_{n,i}(\lambda^*)<0$ whenever $\lambda_i^*\neq \lambda_{i+1}^*$. If instead $\lambda_i^*=\lambda_{i+1}^*$, then clearly
\[
T_{n,i}(\lambda^*)=0.
\]
 Since $\lambda^*$ is not constant, there exists at least one index $i$ such that $\lambda_i^*\neq \lambda_{i+1}^*$, and therefore
\[
\sum_{i=0}^{n-1}T_{n,i}(\lambda^*)<0.
\]\par 

On the other hand,
\[
\sum_{i=0}^{n-1}T_{n,i}(\lambda)=4\inner{\nabla\lambda,\nabla(\Psi(n)\lambda)}-\frac{2}{n}\sum_{i=0}^{n-1}(\lambda_i+\lambda_{i+1})(\lambda_i-\lambda_{i+1})^2.
\]
Since $L=\nabla^*\nabla$, the first term is
\[
4\inner{\nabla\lambda,\nabla(\Psi(n)\lambda)}=4\inner{\lambda,L\Psi(n)\lambda}.
\]
Also,
\[
\nabla(\lambda^2)
=
\bigl(
(\lambda_0+\lambda_1)(\lambda_0-\lambda_1),\dots,
(\lambda_{n-1}+\lambda_0)(\lambda_{n-1}-\lambda_0)
\bigr)^{\mathrm{T}},
\]
and hence
\[
\frac{2}{n}\sum_{i=0}^{n-1}(\lambda_i+\lambda_{i+1})(\lambda_i-\lambda_{i+1})^2 = \frac{2}{n}\inner{\nabla(\lambda^2),\nabla\lambda} = \frac{2}{n}\inner{\lambda^2,L\lambda}.
\]
Therefore, we have
\[
\sum_{i=0}^{n-1}T_{n,i}(\lambda)=K_n(\lambda).
\]
Applying this identity to $\lambda=\lambda^*$ yields
\[
K_n(\lambda^*)<0,
\]
which contradicts the assumption \eqref{Eqn: Kn nonnegative}. Hence \eqref{Eqns: KKT condition absorbing mu} has no solution. The final assertion now follows from Lemma~\ref{Lem: KKT condition absorbing mu}.
\end{proof}
The next step is to estimate the cubic term $\inner{\lambda^2,L\lambda}$ on the Fourier side. To simplify the notation, we first introduce the relevant index sets and combinatorial coefficients. Let $n\geq 4$, and for a triple $(a,b,c)\in \{0,\dots,n-1\}^3$ define
\[
  w_{n,(a,b,c)}=\sin^2(\frac{\pi a}{n})+\sin^2(\frac{\pi b}{n})+\sin^2(\frac{\pi c}{n}).
\]
If $n=2m+1$, define 
\[
    \begin{split}
      &E_{2m+1}^{(1)}=\{(a,b,c):a+b=c,1\leq a\leq b\le c\leq m\},\\
      &E_{2m+1}^{(2)}=\{(a,b,c):a+b+c=2m+1,1\leq a\leq b\le c\leq m\},
    \end{split}
\]
and set
\[
E_{2m+1}:=E_{2m+1}^{(1)}\cup E_{2m+1}^{(2)}.
\] 
If $n=2m$, define 
\[
   \begin{split}
     &E_{2m}^{(1)}=\{(a,b,c):a+b=c,1\leq a\leq b\le c\leq m-1\},\\
     &E_{2m}^{(2)}=\{(a,b,c):a+b+c=2m,1\leq a\leq b\le c\leq m-1\},\\
     &E_{2m}^{(3)}=\{(a,b,m):a+b=m,1\leq a\leq b\le m-1\},
   \end{split}
\]
and set
\[
E_{2m}:=E_{2m}^{(1)}\cup E_{2m}^{(2)}\cup E_{2m}^{(3)}.
\]
Here and below, $E_n$ denotes $E_{2m+1}$ when $n=2m+1$, and $E_{2m}$ when $n=2m$. The following lemma is a simple combinatorial observation, and we state it without proof.
\begin{lemma}\label{Lem: combinatorial description of summation indices}
For $n \ge 4$, define 
\[
\mathcal E_n := \{(r,s,t) \in \{1,\dots,n-1\}^3 : r+s+t = 0 \bmod n\},
\]
and
\[
\overline{\mathcal E}_n := \left\{(a,b,c) \in \{1,\dots,\lfloor n/2 \rfloor \}^3 : a+b+c=n, \, \text{or } \, a+b=c, \, \text{or } \, a+c=b, \, \text{or } \, b+c=a \right\}.
\]
Define the map $\pi:\mathcal{E}_n\to \overline{\mathcal{E}}_n$ by
\[
\pi(r,s,t) \coloneqq (\min(r,n-r),\min(s,n-s),\min(t,n-t)).
\]
Then $\pi(\mathcal E_n) = \overline{\mathcal E}_n$, and $\pi$ is two-to-one. More precisely,
\[
\pi^{-1}(a,b,c) = \begin{cases}
\{(a,b,c), (n-a, n-b, n-c)\}, & a+b+c=n, \\
\{(a,b,n-c), (n-a, n-b,c)\}, & a+b=c, \\
\{(a,n-b,c), (n-a,b,n-c)\}, & a+c=b, \\
\{(n-a,b,c), (a,n-b,n-c)\}, & b+c=a, \\
\end{cases}
\]
For every $(a,b,c) \in \overline{\mathcal E}_n$, its nondecreasing rearrangement is the unique element of $E_n$ obtained by permuting $(a,b,c)$. Let $N_{n,(a,b,c)}$ be the number of distinct ordered triples obtained by permuting $(a,b,c)\in E_n$. Then
\[
N_{2m+1,(a,b,c)} =
\begin{cases}
1, & (a,b,c) \in E_{2m+1}^{(2)} \text{ and } a=b=c, \\
3, & (a,b,c) \in E_{2m+1}^{(2)} \text{ and exactly two of } a,b,c \text{ are equal}, \\
6, & (a,b,c) \in E_{2m+1}^{(2)} \text{ and } a,b,c \text{ are distinct}, \\
3, & (a,b,c) \in E_{2m+1}^{(1)} \text{ and } a=b, \\
6, & (a,b,c) \in E_{2m+1}^{(1)} \text{ and } a<b,
\end{cases}
\]
and
\[
N_{2m,(a,b,c)} =
\begin{cases}
1, & (a,b,c) \in E_{2m}^{(2)} \text{ and } a=b=c, \\
3, & (a,b,c) \in E_{2m}^{(2)} \text{ and exactly two of } a,b,c \text{ are equal}, \\
6, & (a,b,c) \in E_{2m}^{(2)} \text{ and } a,b,c \text{ are distinct}, \\
3, & (a,b,c) \in E_{2m}^{(1)} \cup E_{2m}^{(3)} \text{ and } a=b, \\
6, & (a,b,c) \in E_{2m}^{(1)} \cup E_{2m}^{(3)} \text{ and } a<b,
\end{cases}
\]
\end{lemma}

For $n\geq 4$ and a triple $(a,b,c)\in E_n$, define the weight $\theta_{n,(a,b,c)}=\frac{N_{n,(a,b,c)}}{3}$ for the triple $(a,b,c)$. 
For $x=(x_1,\dots,x_{\floor{n/2}})\in \mathbb{R}_+^{\floor{n/2}}$, set
\[
  Q_{n,3}(x)=\frac{1}{n^3}\sum_{(a,b,c)\in E_{n}}\theta_{n,(a,b,c)} w_{n,(a,b,c)}x_ax_bx_c.
\]
\begin{lemma}\label{Lem: upper bound of cubic term}
  Let $n\geq 4$, and write $n=2m$ or $n=2m+1$ according to parity. For $\lambda\in \mathbb{R}_+^n$, let $\hat{\lambda}=F_n\lambda$ and $\tilde{\lambda}=(\abs{\hat{\lambda}_1},\dots,\abs{\hat{\lambda}_m})\in \mathbb{R}_+^m$. If $n=2m+1$, then we have
\[
\inner{\lambda^2,L\lambda}\leq \frac{16}{n^2}\hat{\lambda}_0\sum_{k=1}^{m}\sin^2(\frac{\pi k}{n})\abs{\hat{\lambda}_k}^2+8n Q_{n,3}(\tilde{\lambda}).
\]
If $n=2m$, then we have
\[
\inner{\lambda^2,L\lambda}\le\frac{16}{n^2} \hat\lambda_0\left(\sum_{k=1}^{m-1}\sin^2(\frac{\pi k}{n})\abs{\hat{\lambda}_k}^2+\frac{1}{2}\abs{\hat{\lambda}_m}^2  \right)+8n Q_{n,3}(\tilde{\lambda}).
\]
\end{lemma}

\begin{proof}
Define
\[
 \mu\coloneqq\lambda-\frac{\hat{\lambda}_0}{n}\mathbf{1}.
\]
By $L\mathbf{1}=0$, we have $L\lambda=L\mu$. Therefore,
\[
\inner{\lambda^2,L\lambda}=\inner{(\frac{\hat{\lambda}_0}{n}\mathbf{1}+\mu)^2,L\mu}=\frac{2\hat{\lambda}_0}{n}\inner{\mu,L\mu}+\inner{\mu^2,L\mu}.
\]
Since $\hat{\mu}_0=0$, $\hat{\mu}_k=\hat{\lambda}_k=\overline{\hat{\lambda}_{n-k}}$ for $k\neq 0$ and
\[
F_nLF_n^{-1}=F_n(2I-S_n-S_n^{-1})F_n^{-1}=\diag\left( 0,4\sin^2(\frac{\pi}{n}),4\sin^2(\frac{2\pi}{n}),\dots,4\sin^2(\frac{(n-1)\pi}{n}) \right),
\]
Parseval's identity yields
\[
\inner{\mu,L\mu}=\frac{4}{n}\sum_{k=1}^{n-1}\sin^2(\frac{\pi k}{n})\abs{\hat{\lambda}_k}^2,
\]
and therefore
\begin{equation}\label{Eqn: 2lambda_av <f,Lmu>}
  \frac{2\hat{\lambda}_0}{n}\inner{\mu,L\mu}=\frac{8}{n^2}\hat{\lambda}_0\sum_{k=1}^{n-1}\sin^2(\frac{\pi k}{n})\abs{\hat{\lambda}_k}^2.
\end{equation}
If $n=2m+1$, by $\hat{\lambda}_{k}=\overline{\hat{\lambda}_{n-k}}$ for real $\lambda$ and $\hat{\mu}_k=\hat{\lambda}_k$ for $k\neq 0$, we get
\begin{equation}\label{Eqn: 2lambda_av <f,Lmu>, odd n}
  \frac{2\hat{\lambda}_0}{n}\inner{\mu,L\mu} = \frac{16}{n^2}\hat{\lambda}_0 \sum_{k=1}^{m}\sin^2(\frac{\pi k}{n})|\hat{\lambda}_k|^2,
\end{equation}
If $n=2m$, similarly 
\begin{equation}\label{Eqn: 2lambda_av <f,Lmu>, even n}
  \frac{2\hat{\lambda}_0}{n}\inner{\mu,L\mu} = \frac{16}{n^2}\hat{\lambda}_0 \left( \sum_{k=1}^{m-1}\sin^2(\frac{\pi k}{n})|\hat{\lambda}_k|^2 +\frac{1}{2}|\hat{\lambda}_m|^2 \right).
\end{equation}\par
It remains to estimate the cubic term $\inner{\mu^2,L\mu}$. Since 
\[
 \mu_j=\frac{1}{n}\sum_{k=0}^{n-1}\omega^{-jk}\hat{\mu}_k,\qquad 0\leq j\leq n-1,
\]
we have
\[
  (\mu^2)_j=\frac{1}{n^2}(\sum_{k=0}^{n-1}\omega^{-jk}\hat{\mu}_k)(\sum_{\ell=0}^{n-1}\omega^{-j\ell}\hat{\mu}_\ell),\qquad 0\leq j\leq n-1,
\]
and hence 
\[
\left( \widehat{\mu^2} \right)_j=\sum_{k=0}^{n-1}\omega^{jk}(\mu^2)_k=\frac{1}{n^2}\sum_{i,\ell=0}^{n-1}\sum_{k=0}^{n-1}\omega^{jk}\omega^{-ki}\omega^{-k\ell}\hat{\mu}_i\hat{\mu}_\ell=\frac{1}{n^2}\sum_{i,\ell=0}^{n-1}\sum_{k=0}^{n-1}\omega^{k(j-i-\ell)}\hat{\mu}_i\hat{\mu}_\ell=\frac{1}{n}\sum_{\substack{i,\ell=0\\ j=i+\ell \bmod n}}^{n-1}\hat{\mu}_i\hat{\mu}_\ell.
\]
Applying Parseval's identity again, we obtain
\[
\begin{split}
  \inner{\mu^2,L\mu}=&\frac{1}{n}\inner{\widehat{\mu^2},\widehat{L\mu}}=\frac{4}{n}\sum_{k=0}^{n-1}\sin^2(\frac{\pi k}{n})\overline{\hat{\mu}_k}\left( \widehat{\mu^2} \right)_k=\frac{4}{n^2}\sum_{k=0}^{n-1}\sin^2(\frac{\pi k}{n})\hat{\mu}_{n-k}\sum_{\substack{\ell,j=0\\\ell+j=k \bmod n}}^{n-1}\hat{\mu}_\ell\hat{\mu}_j\\
  =&\frac{4}{n^2}\sum_{\substack{\ell,j,k=0\\ \ell+j=k \bmod n}}^{n-1}\sin^2(\frac{\pi k}{n})\hat{\mu}_{n-k}\hat{\mu}_\ell\hat{\mu}_j=\frac{4}{n^2}\sum_{\substack{\ell,j,k=0\\ \ell+j+k=0 \bmod n}}^{n-1}\sin^2(\frac{\pi k}{n})\hat{\mu}_\ell\hat{\mu}_j\hat{\mu}_k.
\end{split}
\]
Since the product $\hat{\mu}_\ell\hat{\mu}_j\hat{\mu}_k$ is symmetric in $(\ell,j,k)$, symmetrization gives
\[
\inner{\mu^2,L\mu}=\frac{4}{3n^2}\sum_{\substack{j,k,\ell=0\\j+k+\ell=0 \bmod n}}^{n-1}w_{n,(j,k,\ell)}\hat{\mu}_j\hat{\mu}_k\hat{\mu}_\ell.
\]
 Recalling $\hat{\mu}_0=0$ and $\hat{\mu}_k=\hat{\lambda}_k=\overline{\hat{\lambda}_{n-k}}$ for $k\neq 0$, we obtain

\begin{equation}\label{Eqn: <f^2,Lmu>}
\inner{\mu^2,L\mu}=\frac{4}{3n^2}\sum_{\substack{j,k,\ell=1\\j+k+\ell=0 \bmod n}}^{n-1}w_{n,(j,k,\ell)}\hat{\mu}_j\hat{\mu}_k\hat{\mu}_\ell\leq \frac{4}{3n^2}\sum_{\substack{j,k,\ell=1\\j+k+\ell=0 \bmod n}}^{n-1}w_{n,(j,k,\ell)}\abs{\hat{\lambda}_j}\abs{\hat{\lambda}_k}\abs{\hat{\lambda}_\ell}
\end{equation}
By Lemma \ref{Lem: combinatorial description of summation indices}, we have  
\[
  \begin{split}
    \frac{4}{3n^2}\sum_{\substack{j,k,\ell=1\\j+k+\ell=0 \bmod n}}^{n-1}w_{n,(j,k,\ell)}\abs{\hat{\lambda}_j}\abs{\hat{\lambda}_k}\abs{\hat{\lambda}_\ell}&=\frac{8}{3n^2}\sum_{(a,b,c)\in \overline{\mathcal{E}}_n}w_{n,(a,b,c)}\abs{\hat{\lambda}_a}\abs{\hat{\lambda}_b}\abs{\hat{\lambda}_c}\\
    &=\frac{8}{n^2}\sum_{(a,b,c)\in E_n}\theta_{n,(a,b,c)} w_{n,(a,b,c)}\abs{\hat{\lambda}_a}\abs{\hat{\lambda}_b}\abs{\hat{\lambda}_c}.
  \end{split}
\]

Therefore, we have 
\begin{equation}\label{Ineq: <f^2,Lmu>}
  \begin{split}
  \inner{\mu^2,L\mu}\leq\frac{8}{n^2}\sum_{(a,b,c)\in E_n}\theta_{n,(a,b,c)} w_{n,(a,b,c)}\abs{\hat{\lambda}_a}\abs{\hat{\lambda}_b}\abs{\hat{\lambda}_c}=8nQ_{n,3}(\tilde{\lambda}).\\
\end{split}
\end{equation}
Combining \eqref{Eqn: 2lambda_av <f,Lmu>, odd n} with \eqref{Ineq: <f^2,Lmu>} in the odd case, and \eqref{Eqn: 2lambda_av <f,Lmu>, even n} with \eqref{Ineq: <f^2,Lmu>} in the even case, we obtain the desired estimate of $\inner{\lambda^2,L\lambda}$ in both cases.
\end{proof}

We now derive a lower bound of $K_n(\lambda)$ based on the analysis on the frequency side. For $n=2m+1$ and $x=(x_1,\dots,x_m)\in\mathbb R_+^m$, define
\[
\begin{split}
  &Q_{2m+1,2}(x)=\frac{2}{(2m+1)^2}\sum_{k=1}^m (k-1)\sin^2(\frac{\pi k}{2m+1})x_k^2,\\
  &Q_{2m+1,4}(x)=\frac{2}{(2m+1)^4}\left( \sum_{k=1}^m x_k^2 \right)\left( \sum_{k=1}^m \sin^2(\frac{\pi k}{2m+1})x_k^2 \right).
\end{split}
\]
For $n=2m$ and $x=(x_1,\dots,x_m)\in \mathbb{R}_+^m$, define
\[
\begin{split}
  &Q_{2m,2}(x)=\frac{2}{(2m)^2}\sum_{k=1}^{m-1} (k-1)\sin^2(\frac{\pi k}{2m})x_k^2+\frac{m-1}{(2m)^2}x_m^2,\\
  &Q_{2m,4}(x)=\frac{2}{(2m)^4}\left( \sum_{k=1}^{m-1}x_k^2 +\frac{1}{2}x_m^2\right)\left( \sum_{k=1}^{m-1} \sin^2(\frac{\pi k}{2m})x_k^2 +\frac{1}{2}x_m^2\right).
\end{split}
\]

\begin{theorem}\label{Thm: nonnegativity of Q_2+Q_4-Q_3 implies nonnegativity of K_n}
 Let $n\geq 4$ and $\lambda\in \mathbb R_{+}^n$ satisfy $0<\norm{\lambda}_2^2<n$. Set
\[
\tilde{\lambda}=(\abs{\hat\lambda_1},\dots,\abs{\hat\lambda_{\floor{n/2}}})\in\mathbb R_{+}^{\floor{n/2}}.
\] 
Then
\[
K_n(\lambda)\geq 16(Q_{n,2}(\tilde\lambda)+Q_{n,4}(\tilde\lambda)-Q_{n,3}(\tilde\lambda)).
\]
Consequently, if
\[
Q_{n,2}(x)+Q_{n,4}(x)\ge Q_{n,3}(x) \qquad\text{for all }x\in\mathbb R_{+}^{\floor{n/2}},
\]
then $K_n(\lambda)\ge 0$ for every $\lambda\in\mathbb R_{+}^n$ with $0<\norm{\lambda}_2^2<n$.
\end{theorem}

\begin{proof}
  Since the DFT matrix diagonalizes both $L$ and $\Psi(n)$:
 \[
  F_n\Psi(n)F_n^{-1}=\frac{1}{n}\diag(\psi_n),\qquad F_nLF_n^{-1}=\diag\left( 0,4\sin^2(\frac{\pi}{n}),4\sin^2(\frac{2\pi}{n}),\dots,4\sin^2(\frac{(n-1)\pi}{n}) \right),
\] 
Parseval's identity gives
\begin{equation}\label{Eqn: Fourier side of <lambda,LPsi(n)lambda>}
  \inner{\lambda,L\Psi(n) \lambda}=\frac{1}{n}\inner{\hat{\lambda},F_nLF_n^{-1}F_n\Psi(n)F_n^{-1}\hat{\lambda}}=\frac{4}{n^2}\sum_{k=1}^{n-1}\psi_n(k)\sin^2(\frac{\pi k}{n})\abs{\hat{\lambda}_k}^2.
\end{equation}
We treat the odd and even cases separately.

\noindent\textbf{Case 1: $n=2m+1$.} By $\psi_{n}(k)=\min(k,n-k)$ and $\abs{\hat\lambda_{n-k}}=\abs{\hat\lambda_k}$ for $1\leq k\leq m$, \eqref{Eqn: Fourier side of <lambda,LPsi(n)lambda>} becomes
\[
  \inner{\lambda,L\Psi(n) \lambda}=\frac{8}{n^2}\sum_{k=1}^m k\sin^2(\frac{\pi k}{n})\abs{\hat{\lambda}_k}^2.
\]
By Lemma~\ref{Lem: upper bound of cubic term}, we obtain
\begin{equation}\label{Ineq: Lower bound of K_n, odd n}
  \begin{split}
   K_n(\lambda)&\geq \frac{32}{n^2}\sum_{k=1}^m k\sin^2(\frac{\pi k}{n})\abs{\hat{\lambda}_k}^2-\frac{2}{n}\left( \frac{16}{n^2}\hat{\lambda}_0\sum_{k=1}^{m}\sin^2(\frac{\pi k}{n})\abs{\hat{\lambda}_k}^2+8n Q_{n,3}(\tilde{\lambda}) \right)\\
   &=\frac{32}{n^2}\sum_{k=1}^m \left( k-\frac{\hat{\lambda}_0}{n} \right)\sin^2(\frac{\pi k}{n})\abs{\hat{\lambda}_k}^2-16 Q_{n,3}(\tilde{\lambda}).
 \end{split}
\end{equation}
Next, Parseval's identity and $\hat{\lambda}_k=\overline{\hat{\lambda}_{n-k}}$ give
\[
\frac{\hat\lambda_0^2}{n^2}+\frac{2}{n^2}\sum_{k=1}^{m}\abs{\hat{\lambda}_k}^2=\frac{1}{n^2}\sum_{k=0}^{n-1}\abs{\hat\lambda_k}^2=\frac{\norm{\lambda}_2^2}{n}<1.
\]
Since $\lambda\in\mathbb R_+^n$ and $\norm{\lambda}_2^2<n$, by the Cauchy-Schwarz inequality we have $\hat{\lambda}_0<n$. It follows that
\[
1-\frac{\hat\lambda_0}{n} = \frac{1-\frac{\hat\lambda_0^2}{n^2}}{1+\frac{\hat\lambda_0}{n}}>\frac{\frac{2}{n^2}\sum_{k=1}^{m}\abs{\hat{\lambda}_k}^2}{1+\frac{\hat\lambda_0}{n}}>\frac{\sum_{k=1}^{m}\abs{\hat{\lambda}_k}^2}{n^2}.
\]
Therefore, we have
\[
  \begin{split}
    \frac{32}{n^2}\sum_{k=1}^m \left( k-\frac{\hat{\lambda}_0}{n} \right)\sin^2(\frac{\pi k}{n})\abs{\hat{\lambda}_k}^2=&\frac{32}{n^2}\sum_{k=1}^m \left( k-1+1-\frac{\hat{\lambda}_0}{n} \right)\sin^2(\frac{\pi k}{n})\abs{\hat{\lambda}_k}^2\\
    \geq &\frac{32}{n^2}\sum_{k=1}^m (k-1)\sin^2(\frac{\pi k}{n})\abs{\hat{\lambda}_k}^2+\frac{32}{n^4}\left( \sum_{k=1}^m\abs{\hat{\lambda}_k}^2 \right)\sum_{k=1}^m\sin^2(\frac{\pi k}{n})\abs{\hat{\lambda}_k}^2\\
    =&16Q_{n,2}(\tilde{\lambda})+16Q_{n,4}(\tilde{\lambda}).
  \end{split}
\]
Combining this with \eqref{Ineq: Lower bound of K_n, odd n}, we conclude that
\[
  K_n(\lambda)\geq 16Q_{n,2}(\tilde{\lambda})+16Q_{n,4}(\tilde{\lambda})-16 Q_{n,3}(\tilde{\lambda}),
\]
which proves the theorem for odd $n=2m+1$.\par 

\noindent\textbf{Case 2: $n=2m$.} In this case, the middle Fourier coefficient $\hat{\lambda}_m$ occurs only once in \eqref{Eqn: Fourier side of <lambda,LPsi(n)lambda>}. By $\psi_{n}(k)=\min(k,n-k)$ and $\abs{\hat\lambda_{n-k}}=\abs{\hat\lambda_k}$ for $1\leq k\leq m$, we have
\[
  \begin{split}
    \inner{\lambda,L\Psi(n) \lambda}=\frac{8}{n^2}\sum_{k=1}^{m-1} k\sin^2(\frac{\pi k}{n})\abs{\hat{\lambda}_k}^2+\frac{2}{n}\abs{\hat{\lambda}_{m}}^2.
  \end{split}
\]
By Lemma~\ref{Lem: upper bound of cubic term},  we obtain
\begin{equation}\label{Ineq: Lower bound of K_n, even n}
  \begin{split}
   K_n(\lambda)\geq &\frac{32}{n^2}\sum_{k=1}^{m-1} k\sin^2(\frac{\pi k}{n})\abs{\hat{\lambda}_k}^2+\frac{8}{n}\abs{\hat{\lambda}_{m}}^2-\frac{2}{n}\left( \frac{16}{n^2}\hat{\lambda}_0\left( \sum_{k=1}^{m-1}\sin^2(\frac{\pi k}{n})\abs{\hat{\lambda}_k}^2+\frac{1}{2}\abs{\hat{\lambda}_{m}}^2 \right)+8n Q_{n,3}(\tilde{\lambda}) \right)\\
   =&\frac{32}{n^2}\sum_{k=1}^{m-1} \left( k-\frac{\hat{\lambda}_0}{n} \right)\sin^2(\frac{\pi k}{n})\abs{\hat{\lambda}_k}^2+\left( \frac{8}{n}-\frac{16\hat{\lambda}_0}{n^3} \right)\abs{\hat{\lambda}_{m}}^2-16 Q_{n,3}(\tilde{\lambda}).
 \end{split}
\end{equation}
Next, Parseval's identity and the reality of $\lambda$ give
\[
\frac{1}{n^2}\hat\lambda_0^2+\frac{2}{n^2}\sum_{k=1}^{m-1}\abs{\hat{\lambda}_k}^2+\frac{1}{n^2}\abs{\hat{\lambda}_{m}}^2=\frac{1}{n^2}\sum_{k=0}^{n-1}\abs{\hat\lambda_k}^2=\frac{\norm{\lambda}_2^2}{n}<1.
\]
Since $\lambda\in\mathbb R_+^n$ and $\norm{\lambda}_2^2<n$, we have $\hat{\lambda}_0<n$. It follows that
\[
1-\frac{\hat\lambda_0}{n} = \frac{1-\frac{\hat\lambda_0^2}{n^2}}{1+\frac{\hat\lambda_0}{n}}>\frac{\frac{2}{n^2}\sum_{k=1}^{m-1}\abs{\hat{\lambda}_k}^2+\frac{1}{n^2}\abs{\hat{\lambda}_m}^2}{1+\frac{\hat\lambda_0}{n}}>\frac{1}{n^2}\left( \sum_{k=1}^{m-1}\abs{\hat{\lambda}_k}^2+\frac{1}{2}\abs{\hat{\lambda}_m}^2 \right).
\]
Using this estimate, we obtain
\[
  \begin{split}
    &\frac{32}{n^2}\sum_{k=1}^{m-1} \left( k-\frac{\hat{\lambda}_0}{n} \right)\sin^2(\frac{\pi k}{n})\abs{\hat{\lambda}_k}^2=\frac{32}{n^2}\sum_{k=1}^{m-1} \left( k-1+1-\frac{\hat{\lambda}_0}{n} \right)\sin^2(\frac{\pi k}{n})\abs{\hat{\lambda}_k}^2\\
    &\qquad \geq \frac{32}{n^2}\sum_{k=1}^{m-1} (k-1)\sin^2(\frac{\pi k}{n})\abs{\hat{\lambda}_k}^2+\frac{32}{n^4}\left( \sum_{k=1}^{m-1}\abs{\hat{\lambda}_k}^2+\frac{1}{2}\abs{\hat{\lambda}_m}^2 \right)\sum_{k=1}^{m-1}\sin^2(\frac{\pi k}{n})\abs{\hat{\lambda}_k}^2,
  \end{split}
\]
and 
\[
  \left( \frac{8}{n}-\frac{16\hat{\lambda}_0}{n^3} \right)\abs{\hat{\lambda}_{m}}^2=\frac{16}{n^2}\left( m-1+1-\frac{\hat{\lambda}_0}{n} \right)\abs{\hat{\lambda}_{m}}^2> \frac{16}{n^2}(m-1)\abs{\hat{\lambda}_{m}}^2+\frac{16}{n^4}\abs{\hat{\lambda}_m}^2\left( \sum_{k=1}^{m-1}\abs{\hat{\lambda}_k}^2+\frac{1}{2}\abs{\hat{\lambda}_m}^2 \right).
\]
Adding the last two inequalities gives 
\[
  \begin{split}
    &\frac{32}{n^2}\sum_{k=1}^{m-1} \left( k-\frac{\hat{\lambda}_0}{n} \right)\sin^2(\frac{\pi k}{n})\abs{\hat{\lambda}_k}^2+\left( \frac{8}{n}-\frac{16\hat{\lambda}_0}{n^3} \right)\abs{\hat{\lambda}_{m}}^2\\
    &\qquad \geq \frac{32}{n^2}\sum_{k=1}^{m-1} (k-1)\sin^2(\frac{\pi k}{n})\abs{\hat{\lambda}_k}^2+\frac{32}{n^4}\left( \sum_{k=1}^{m-1}\abs{\hat{\lambda}_k}^2+\frac{1}{2}\abs{\hat{\lambda}_m}^2 \right)\sum_{k=1}^{m-1}\sin^2(\frac{\pi k}{n})\abs{\hat{\lambda}_k}^2\\
     &\qquad\quad +\frac{16}{n^2}(m-1)\abs{\hat{\lambda}_{m}}^2+\frac{16}{n^4}\abs{\hat{\lambda}_m}^2\left( \sum_{k=1}^{m-1}\abs{\hat{\lambda}_k}^2+\frac{1}{2}\abs{\hat{\lambda}_m}^2 \right)\\
     &\qquad=16Q_{n,2}(\tilde{\lambda})+16Q_{n,4}(\tilde{\lambda}).
  \end{split}
\]
Combining this with \eqref{Ineq: Lower bound of K_n, even n}, we conclude that
\[
  K_n(\lambda)\geq 16Q_{n,2}(\tilde{\lambda})+16Q_{n,4}(\tilde{\lambda})-16 Q_{n,3}(\tilde{\lambda}),
\]
which proves the theorem for even $n=2m$. 
\end{proof}
The following lemma is one of the most technical parts of the paper.
\begin{lemma}\label{Lem: nonnegativity of Q_2+Q_4-Q_3}
Let $n\ge 4$. Then we have
\[
Q_{n,2}(x)+Q_{n,4}(x)-Q_{n,3}(x)\ge 0,\qquad x\in\mathbb R_+^{\floor{n/2}}.
\]
\end{lemma}

The proof of Lemma \ref{Lem: nonnegativity of Q_2+Q_4-Q_3} depends on the parity of $n$. The odd case is treated in Section \ref{Sec: Proof of key lemma, odd n}, where the case $n=5$ is handled separately and the general odd case $n=2m+1\ge 7$ is established in Theorem \ref{Thm: Choice of theta for odd n geq 7}. The even case $n=2m\ge 4$ is treated in Theorem \ref{Thm: Choice of theta for even n geq 4} in Section \ref{Sec: Proof of key lemma, even n}. As an immediate consequence, we obtain the main theorem.
\begin{proof}[Proof of Theorem \ref{Thm: Log Sobolev inequality n>=4}]
  By Lemma \ref{Lem: nonnegativity of Q_2+Q_4-Q_3}, we have
\[
Q_{n,2}(x)+Q_{n,4}(x)-Q_{n,3}(x)\ge 0
\qquad \text{for all } x\in\mathbb R_+^{\floor{n/2}}.
\]
Theorem \ref{Thm: nonnegativity of Q_2+Q_4-Q_3 implies nonnegativity of K_n}
therefore implies that
\[
K_n(\lambda)\ge 0
\qquad \text{for all } \lambda\in\mathbb R_+^n \text{ with } 0<\norm{\lambda}_2^2<n.
\]
Applying Theorem \ref{Thm: nonnegativity of K_n implies n--LSI}, we conclude
that the $n$--LSI holds for every $n\ge 4$.
\end{proof}

\begin{proof}[Proof of Corollary \ref{Coro: LSI on circle}]
  Let $P_t^{(n)}$ denote the Poisson-like semigroup on $\mathbb{Z}_n$ in \eqref{Eqn: Formula of P_t}, and let $\mathcal{P}_t$ denote the Poisson semigroup on $\mathbb{S}^1$.
  A standard series expansion of $\norm{\mathcal{P}_t(1+\epsilon\sin\theta)}_q$ and $\norm{1+\epsilon \sin\theta}_p$ at $\epsilon=0$ shows that the condition
\[
  t_{p,q}\ge \tfrac{1}{2}\log\big(\tfrac{q-1}{p-1}\big)
\]
is necessary. Thus it remains to prove that this condition is sufficient. \par 
Fix $t\geq\tfrac{1}{2}\log\big(\tfrac{q-1}{p-1}\big)$. By the density of trigonometric polynomials in $L_p$ and $L_p$-contractivity of $\mathcal{P}_t$, it is sufficient to treat the case where $f$ is a trigonometric polynomial. Write $f$ as
\[
  f(\theta)=\sum_{\abs{m}\leq M}a_m e^{i m\theta}.
\]
Choose $n>2M$. Let $\omega=e^{2\pi i/n}$, and define $f_n:\mathbb{Z}_{n}\to \mathbb{C}$ by
\[
  f_n(j)=f(2\pi j/n)=\sum_{\abs{m}\leq M}a_m\omega^{mj}=\sum_{m=0}^M a_m\omega^{mj}+\sum_{m=1}^M a_{-m}\omega^{(n-m)j}, \qquad 0\leq j\leq n-1.
\]
Since $n>2M$, for every $1\le m\le M$, we have
\[
  \psi_{n}(m)=\psi_n(n-m)=m.
\]
Therefore
\[
\begin{aligned}
  (P_t^{(n)}f_n)(j)=\sum_{m=0}^M e^{-tm}a_m\omega^{mj}+\sum_{m=1}^M e^{-tm}a_{-m}\omega^{(n-m)j}=\sum_{\abs{m}\leq M}e^{-t\abs{m}}a_m\omega^{mj} =(\mathcal{P}_t f)(2\pi j/n).
\end{aligned}
\]
By Theorem \ref{Thm: Hypercontractivity on n}, we have
\[
  \left( \frac{1}{n}\sum_{j=0}^{n-1}
    \abs*{\mathcal{P}_t f(2\pi j/n)}^q \right)^{\frac{1}{q}}\le\left( \frac{1}{n}\sum_{j=0}^{n-1} \abs{f(2\pi j/n)}^p \right)^{1/p}.
\]
Letting $n\to\infty$, the Riemann sums on both sides converge and thus
\[
  \norm{\mathcal{P}_t f}_{L_q(\mathbb{S}^1)}\le \norm{f}_{L_p(\mathbb{S}^1)}
\]
for every trigonometric polynomial $f$. Finally, the LSI follows from Gross's theorem~\cite{MR420249}.
\end{proof}

\begin{proof}[Proof of Corollary~\ref{Coro: LSI for simple random walk}]  Let $\frac{1}{\alpha}$ denote the log-Sobolev constant of the simple random walk on \(\mathbb{Z}_n\). By \cite[Theorem 1]{MR1990534}, we have $2\alpha\le \lambda$, where $\lambda=1-\cos(\frac{2\pi}{n})$ is the spectral gap. So any admissible constant in \eqref{Ineq: LSI for simple random walk} must be at least
\[
\frac{1}{\alpha}\ge \frac{2}{1-\cos(\frac{2\pi}{n})}.
\]
Therefore it suffices to prove \eqref{Ineq: LSI for simple random walk}.\par

Let $f\in L_2^+(\mathbb{Z}_n,\mu_n)$, and write
\[
f=\sum_{j=0}^{n-1} a_j\chi_j, \qquad \chi_j(x)=e^{\frac{2\pi i jx}{n}}.
\]
Since 
\[
  (A_{\psi_n}\chi_j)(x)=\psi_n(j)\chi_j(x),\qquad (K\chi_j)(x)=\frac{1}{2}(\chi_j(x+1)+\chi_j(x-1))=\cos\!\left(\frac{2\pi j}{n}\right)\chi_j(x),
\]
we have
\[
\inner{ f,A_{\psi_n}f}_{L^2(\mathbb{Z}_n,\mu_n)}=\sum_{j=0}^{n-1}\psi_n(j)\abs{a_j}^2
\]
and
\[
  \inner{f,(I-K)f}_{L^2(\mathbb{Z}_n,\mu_n)}=\sum_{j=0}^{n-1}\left( 1-\cos(\frac{2\pi j}{n}) \right)\abs{a_j}^2.
\]
Therefore it remains to show that
\begin{equation}\label{Ineq: comparing eigenvalues}
  \psi_n(j)\leq \frac{1- \cos(\frac{2\pi j}{n})}{1-\cos(\frac{2\pi}{n})},\qquad 0\le j\le n-1.
\end{equation}
Since $\psi_n(j)=\min(j,n-j)\leq \frac{n}{2}$ and $\cos(\frac{2\pi j}{n})=\cos(\frac{2\pi (n-j)}{n})$, it suffices to prove 
\[
j\left( 1-\cos(\frac{2\pi}{n}) \right)\leq 1- \cos(\frac{2\pi j}{n}),\qquad 0\le j\le \floor{\frac{n}{2}}.
\]
Define
\[
\varphi(x)=1-\cos\!\left(\frac{2\pi x}{n}\right),
\qquad 0\le x\le \frac{n}{2}.
\]
For \(a,b\ge 0\) with \(a+b\le \frac{n}{2}\), we have
\[
\varphi(a+b)-\varphi(a)-\varphi(b)=4\sin(\frac{\pi a}{n})\sin(\frac{\pi b}{n})\cos(\frac{\pi(a+b)}{n})\ge 0.
\]
Hence $\varphi$ is superadditive on $[0,\frac{n}{2}]$. By induction,
\[
1- \cos(\frac{2\pi j}{n})=\varphi(j)\ge j\,\varphi(1)
= j\left(1-\cos\!\left(\frac{2\pi}{n}\right)\right),
\qquad 0\le j\le \floor{\frac{n}{2}}.
\]
This proves \eqref{Ineq: comparing eigenvalues}. Applying Theorem~\ref{Thm: Log Sobolev inequality n>=4}, we obtain
\[
\int_{\mathbb{Z}_n} f^2\log f^2\,d\mu_n
-\|f\|_2^2\log \|f\|_2^2
\le
2\langle f,A_{\psi_n}f\rangle_{L^2(\mathbb{Z}_n,\mu_n)}
\le
\frac{2}{1-\cos\!\left(\frac{2\pi}{n}\right)}
\langle f,(I-K)f\rangle_{L^2(\mathbb{Z}_n,\mu_n)}.
\]
Thus inequality \eqref{Ineq: LSI for simple random walk} holds, and the constant is optimal.
\end{proof}

\section{Proof of Lemma \ref{Lem: nonnegativity of Q_2+Q_4-Q_3} for $\mathbb{Z}_{2m+1}$ with $2m+1\geq 5$}\label{Sec: Proof of key lemma, odd n}

We first prove Lemma \ref{Lem: nonnegativity of Q_2+Q_4-Q_3} in the special case $n=5$ by direct computation. The general odd case $n=2m+1\geq 7$ is considerably more involved.

\begin{proof}[Proof of Lemma \ref{Lem: nonnegativity of Q_2+Q_4-Q_3} in the case $n=5$]
A direct computation gives
\[
    Q_{5,2}(x)=\frac{2}{5^2}\sum_{k=1}^2(k-1)\sin^2(\frac{\pi k}{5})x_k^2=\frac{5+\sqrt{5}}{100}x_2^2,
  \]
  \[
    Q_{5,4}(x)=\frac{2\sum_{k=1}^2 x_k^2}{5^4}\sum_{k=1}^2\sin^2(\frac{\pi k}{5})x_k^2=\frac{x_1^2+x_2^2}{2500}((5-\sqrt{5})x_1^2+(5+\sqrt{5})x_2^2),
  \]
  and
  \[
   Q_{5,3}(x)=\frac{1}{5^3}\left(w_{5,(1,2,2)}x_1x_2x_2+w_{5,(1,1,2)}x_1x_1x_2 \right)=\frac{1}{1000}((15+\sqrt{5})x_1x_2^2+(15-\sqrt{5})x_1^2x_2).
  \]
Expanding the difference yields
\[
 \begin{split}
   &Q_{5,2}(x)+Q_{5,4}(x)-Q_{5,3}(x)\\
   &\qquad =\frac{1}{2500}\left((5-\sqrt5)x_1^4+10x_1^2x_2^2+(5+\sqrt5)x_2^4+25(5+\sqrt5)x_2^2-\frac52(15-\sqrt5)x_1^2x_2-\frac52(15+\sqrt5)x_1x_2^2 \right).
 \end{split}
\]
Now complete the square in the bracket:
\[
\begin{split}
  &(5-\sqrt5)x_1^4+10x_1^2x_2^2+(5+\sqrt5)x_2^4+25(5+\sqrt5)x_2^2-\frac52(15-\sqrt5)x_1^2x_2-\frac52(15+\sqrt5)x_1x_2^2 \\
  &\qquad =(5-\sqrt{5}) \left(x_1^2-\frac{5}{8}(7+\sqrt{5}) x_2\right)^2+10 \left(x_1-\frac{15+\sqrt{5}}{8} \right)^2 x_2^2+(\sqrt{5}+5) x_2^4+\frac{25}{16}(7+9 \sqrt{5}) x_2^2.
\end{split}
\]  
It follows that
\[
Q_{5,2}(x)+Q_{5,4}(x)-Q_{5,3}(x)\ge 0.
\]
This proves Lemma \ref{Lem: nonnegativity of Q_2+Q_4-Q_3} for $n=5$.
\end{proof}
We now turn to the case $n=2m+1\ge 7$. For a triple $(a,b,c)$, set
\[
A_{n,(a,b,c)}^{(c)}\coloneqq\frac{w_{n,(a,b,c)}^2}{4\left(\sin^2(\frac{\pi a}{n})+\sin^2(\frac{\pi b}{n})\right)},
\qquad
A_{n,(a,b,c)}^{(b)}\coloneqq\frac{w_{n,(a,b,c)}^2}{4\left(\sin^2(\frac{\pi a}{n})+\sin^2(\frac{\pi c}{n})\right)}.
\]
Hereafter, for a triple $(a,b,c)$, the superscripts indicate whether the second or third entry of the triple is being singled out. For example, if $(a,b,c)=(2,3,4)$, then $A_{n,(a,b,c)}^{(c)}=A_{n,(2,3,4)}^{(c)}=A_{n,(2,3,4)}^{(4)}$, while $A_{n,(a,b,c)}^{(b)}=A_{n,(2,3,4)}^{(b)}=A_{n,(2,3,4)}^{(3)}$.\par 
The following proposition gives a sufficient criterion for the estimate
\[
Q_{n,2}(x)+Q_{n,4}(x)-Q_{n,3}(x)\ge 0
\]
for the odd case $n=2m+1\geq 7$.

\begin{proposition}\label{Prop: Decomposition of theta_e, odd n}
Let $n=2m+1\ge 7$. Suppose that for every $(a,b,c)\in E_n^{(1)}\cup E_n^{(2)}$, one can choose nonnegative numbers $\theta_{n,(a,b,c)}^{(b)}$ and $\theta_{n,(a,b,c)}^{(c)}$ such that
\[
\theta_{n,(a,b,c)}=\theta_{n,(a,b,c)}^{(b)}+\theta_{n,(a,b,c)}^{(c)}.
\]
For $1\le u\le v\le m$, define
\[
L_{uv}\coloneqq \sum_{(u,v,c)\in E_n^{(1)}\cup E_n^{(2)}}
\theta_{n,(u,v,c)}^{(c)} + \sum_{(u,b,v)\in E_n^{(1)}\cup E_n^{(2)}} \theta_{n,(u,b,v)}^{(b)},
\]
and for $1\le j\le m$, define
\[
R_j\coloneqq \sum_{(a,b,j)\in E_n^{(1)}\cup E_n^{(2)}} \theta_{n,(a,b,j)}^{(c)}A_{n,(a,b,j)}^{(c)} + \sum_{(a,j,c)\in E_n^{(1)}\cup E_n^{(2)}} \theta_{n,(a,j,c)}^{(b)}A_{n,(a,j,c)}^{(b)}.
\]
Assume moreover that
\begin{enumerate}
  \item $L_{uv}\le 2$ for all $1\le u<v\le m$;
  \item $L_{uu}\le 1$ for all $1\le u\le m$;
  \item $R_j\le 2(j-1)\sin^2(\frac{\pi j}{n})$ for all $1\le j\le m$.
\end{enumerate}
Then
\[
Q_{n,2}(x)+Q_{n,4}(x)-Q_{n,3}(x)\ge 0,
\qquad \text{for all }x\in \mathbb{R}_+^m.
\]
\end{proposition}

\begin{proof}
We have
\[
n^3Q_{n,3}(x)
=
\sum_{(a,b,c)\in E_n^{(1)}\cup E_n^{(2)}}
\bigl(\theta_{n,(a,b,c)}^{(b)}+\theta_{n,(a,b,c)}^{(c)}\bigr)w_{n,(a,b,c)}x_ax_bx_c.
\]
Fix $(a,b,c)\in E_n^{(1)}\cup E_n^{(2)}$. By the inequality $yz\le \frac{y^2}{n}+\frac{nz^2}{4}$, we obtain
\begin{equation}\label{Ineq: estimate of theta w x for c}
  \theta_{n,(a,b,c)}^{(c)}w_{n,(a,b,c)} x_a x_b x_c \le \frac{1}{n}\theta_{n,(a,b,c)}^{(c)}\left( \sin^2(\frac{\pi a}{n})+\sin^2(\frac{\pi b}{n}) \right)x_a^2x_b^2 +n\theta_{n,(a,b,c)}^{(c)}A_{n,(a,b,c)}^{(c)}x_c^2,
\end{equation}
and similarly,
\begin{equation}\label{Ineq: estimate of theta w x for b}
  \theta_{n,(a,b,c)}^{(b)}w_{n,(a,b,c)} x_a x_b x_c \le \frac{1}{n}\theta_{n,(a,b,c)}^{(b)}\left( \sin^2(\frac{\pi a}{n})+\sin^2(\frac{\pi c}{n}) \right)x_a^2x_c^2+n\theta_{n,(a,b,c)}^{(b)}A_{n,(a,b,c)}^{(b)}x_b^2.
\end{equation}
Fix $(a,b)=(u,v)$ and sum the first term on the right-hand side of \eqref{Ineq: estimate of theta w x for c} over the index $c$. Also fix $(a,c)=(u,v)$ and sum the first term on the right-hand side of \eqref{Ineq: estimate of theta w x for b} over the index $b$. This gives
\[
  \begin{split}
    &\sum_{(u,v,c)\in E_n^{(1)}\cup E_n^{(2)}}\frac{1}{n}\theta_{n,(u,v,c)}^{(c)}\left(\sin^2(\frac{\pi u}{n})+\sin^2(\frac{\pi v}{n})\right)x_u^2x_v^2+\sum_{(u,b,v)\in E_n^{(1)}\cup E_n^{(2)}}\frac{1}{n}\theta_{n,(u,b,v)}^{(b)}\left(\sin^2(\frac{\pi u}{n})+\sin^2(\frac{\pi v}{n})\right)x_u^2x_v^2\\
    &\qquad=\frac{1}{n}\left(\sin^2(\frac{\pi u}{n})+\sin^2(\frac{\pi v}{n})\right)x_u^2x_v^2\left( \sum_{(u,v,c)\in E_n^{(1)}\cup E_n^{(2)}}\theta_{n,(u,v,c)}^{(c)}+\sum_{(u,b,v)\in E_n^{(1)}\cup E_n^{(2)}}\theta_{n,(u,b,v)}^{(b)} \right)\\
    &\qquad =\frac{1}{n}L_{uv}\left(\sin^2(\frac{\pi u}{n})+\sin^2(\frac{\pi v}{n})\right)x_u^2x_v^2.
  \end{split}
\]
Similarly, fixing $c=j$ and summing the second term on the right-hand side of \eqref{Ineq: estimate of theta w x for c} over $(a,b)$, and fixing $b=j$ and summing the second term on the right-hand side of \eqref{Ineq: estimate of theta w x for b} over $(a,c)$, we obtain
\[
  \begin{split}
    &\sum_{(a,b,j)\in E_n^{(1)}\cup E_n^{(2)}}n\theta_{n,(a,b,j)}^{(c)}A_{n,(a,b,j)}^{(c)}x_j^2+\sum_{(a,j,c)\in E_n^{(1)}\cup E_n^{(2)}}n\theta_{n,(a,j,c)}^{(b)}A_{n,(a,j,c)}^{(b)}x_j^2\\
    &\qquad=n x_j^2\left( \sum_{(a,b,j)\in E_n^{(1)}\cup E_n^{(2)}}\theta_{n,(a,b,j)}^{(c)}A_{n,(a,b,j)}^{(c)}+\sum_{(a,j,c)\in E_n^{(1)}\cup E_n^{(2)}}\theta_{n,(a,j,c)}^{(b)}A_{n,(a,j,c)}^{(b)} \right)\\
    &\qquad =n x_j^2R_j.
  \end{split}
\]
Therefore, after collecting terms by the pair $(u,v)$ with $u\le v$ and by the index $j$, we obtain
\[
n^3 Q_{n,3}(x)
\le
\frac{1}{n}\sum_{1\le u\le v\le m}L_{uv}\left(\sin^2(\frac{\pi u}{n})+\sin^2(\frac{\pi v}{n})\right)x_u^2x_v^2+n\sum_{j=1}^m R_jx_j^2.
\]
By assumptions (1) and (2),
\[
\begin{split}
    &\sum_{1\le u\le v\le m}L_{uv}\left(\sin^2(\frac{\pi u}{n})+\sin^2(\frac{\pi v}{n})\right)x_u^2x_v^2\\
    &\qquad=\sum_{u=1}^m 2L_{uu}\sin^2(\frac{\pi u}{n}) x_u^4+\sum_{1\le u<v\le m}L_{uv}\left(\sin^2(\frac{\pi u}{n})+\sin^2(\frac{\pi v}{n})\right)x_u^2x_v^2\\
&\qquad\le\sum_{u=1}^m 2\sin^2(\frac{\pi u}{n}) x_u^4+\sum_{1\le u<v\le m}2\left(\sin^2(\frac{\pi u}{n})+\sin^2(\frac{\pi v}{n})\right)x_u^2x_v^2\\
&\qquad=\sum_{u=1}^m\sum_{v=1}^m \left(\sin^2(\frac{\pi u}{n})+\sin^2(\frac{\pi v}{n})\right)x_u^2x_v^2\\
&\qquad=2\sum_{v=1}^m x_v^2\sum_{u=1}^m\sin^2(\frac{\pi u}{n})x_u^2\\
&\qquad=n^4 Q_{n,4}(x).
  \end{split}
\]
By assumption (3), we also have
\[
\sum_{j=1}^m R_jx_j^2 \le 2\sum_{j=1}^m (j-1)\sin^2(\frac{\pi j}{n})x_j^2 = n^2 Q_{n,2}(x).
\]
Hence, we have,
\[
n^3Q_{n,3}(x) \le n^3Q_{n,4}(x)+n^3Q_{n,2}(x).
\]
This proves the proposition.
\end{proof}

We first derive an upper bound for $A_{n,(a,b,c)}^{(c)}$ that depends only on the third entry $c$.
\begin{lemma}[Upper bounds for $A_{n,(a,b,c)}^{(c)}$]\label{Lem: Upper bound of A}
For $n\ge 4$, let $(a,b,c)\in E_n^{(1)}\cup E_n^{(2)}$. Then
\[
  A_{n,(a,b,c)}^{(c)}\leq\begin{cases}
    \sin^2(\frac{\pi c}{n})\left( 1+\frac{\cos^2(\frac{\pi c}{n})}{4} \right),&(a,b,c)\in E_n^{(1)},\\
    \sin^2(\frac{\pi c}{n})\left( 1+\frac{\cos^2(\frac{\pi c}{n})}{2} \right),&(a,b,c)\in E_n^{(2)}.
  \end{cases}
  \]
\end{lemma}

\begin{proof}

Since $(a,b,c)\in E_n^{(1)}\cup E_n^{(2)}$, we have either $a+b=c$ or $a+b+c=n$. In both cases,
\[
\sin^2(\frac{\pi c}{n})=\sin^2(\frac{\pi(a+b)}{n}).
\]
Hence
\[
  \sin^2(\frac{\pi a}{n})+\sin^2(\frac{\pi b}{n})-\sin^2(\frac{\pi c}{n})=\pm 2\sin(\frac{\pi a}{n})\sin(\frac{\pi b}{n})\cos \left( \frac{\pi c}{n} \right),
\]
and therefore
\[
 \left( \sin^2(\frac{\pi a}{n})+\sin^2(\frac{\pi b}{n})-\sin^2(\frac{\pi c}{n}) \right)^2=4\sin^2(\frac{\pi a}{n})\sin^2(\frac{\pi b}{n})\cos^2(\frac{\pi c}{n}).
\]
It follows that
\[
\begin{split}
      &\left( \sin^2(\frac{\pi a}{n})+\sin^2(\frac{\pi b}{n})+\sin^2(\frac{\pi c}{n}) \right)^2\\
      &\qquad=\left( \sin^2(\frac{\pi a}{n})+\sin^2(\frac{\pi b}{n})-\sin^2(\frac{\pi c}{n}) \right)^2+4\sin^2(\frac{\pi c}{n})\left(\sin^2(\frac{\pi a}{n})+\sin^2(\frac{\pi b}{n})\right)\\
      &\qquad=4\sin^2(\frac{\pi a}{n})\sin^2(\frac{\pi b}{n})\cos^2(\frac{\pi c}{n})+4\sin^2(\frac{\pi c}{n})\left(\sin^2(\frac{\pi a}{n})+\sin^2(\frac{\pi b}{n})\right)\\
      &\qquad=4\left(\sin^2(\frac{\pi a}{n})+\sin^2(\frac{\pi b}{n})\right)\left( \sin^2(\frac{\pi c}{n})+\cos^2(\frac{\pi c}{n})\frac{\sin^2(\frac{\pi a}{n})\sin^2(\frac{\pi b}{n})}{\sin^2(\frac{\pi a}{n})+\sin^2(\frac{\pi b}{n})} \right).
  \end{split}
\]
Consequently,
\begin{equation}\label{Eqn: formula of A_e^{(c)}}
 A_{n,(a,b,c)}^{(c)}=\frac{(\sin^2(\frac{\pi a}{n})+\sin^2(\frac{\pi b}{n})+\sin^2(\frac{\pi c}{n}))^2}{4\left(\sin^2(\frac{\pi a}{n})+\sin^2(\frac{\pi b}{n})\right)}=\sin^2(\frac{\pi c}{n})+\cos^2(\frac{\pi c}{n})\frac{\sin^2(\frac{\pi a}{n})\sin^2(\frac{\pi b}{n})}{\sin^2(\frac{\pi a}{n})+\sin^2(\frac{\pi b}{n})}.
\end{equation}
Using $\frac{xy}{x+y}\leq \frac{x+y}{4}$ for $x,y>0$, we obtain
\begin{equation}\label{Eqn: estimation of A_e^{(c)} for sigma and delta}
  A_{n,(a,b,c)}^{(c)}\le \sin^2(\frac{\pi c}{n})+\frac{\cos^2(\frac{\pi c}{n})}{4}\left(\sin^2(\frac{\pi a}{n})+\sin^2(\frac{\pi b}{n})\right).
\end{equation}

If $(a,b,c)\in E_n^{(1)}$, then $a+b=c\le m$. Hence $\cos(\frac{\pi c}{n})\ge 0$, and so
\[
\sin^2(\frac{\pi a}{n})+\sin^2(\frac{\pi b}{n})-\sin^2(\frac{\pi c}{n})=-2\sin(\frac{\pi a}{n})\sin(\frac{\pi b}{n})\cos(\frac{\pi (a+b)}{n})\leq 0.
\]
Thus $\sin^2(\frac{\pi a}{n})+\sin^2(\frac{\pi b}{n})\le \sin^2(\frac{\pi c}{n})$. Substituting this into \eqref{Eqn: estimation of A_e^{(c)} for sigma and delta} gives
\[
A_{n,(a,b,c)}^{(c)}\le \sin^2(\frac{\pi c}{n})\left( 1+\frac{\cos^2(\frac{\pi c}{n})}{4} \right),
\]
which proves the $E_n^{(1)}$-case.

If $(a,b,c)\in E_n^{(2)}$, then $1\le a\le b\le c\le m$. Since $\sin x$ is
increasing on $[0,\frac{\pi}{2}]$, we have $\sin^2(\frac{\pi a}{n})\le \sin^2(\frac{\pi c}{n}),\sin^2(\frac{\pi b}{n})\le \sin^2(\frac{\pi c}{n})$, and hence $\sin^2(\frac{\pi a}{n})+\sin^2(\frac{\pi b}{n})\le 2\sin^2(\frac{\pi c}{n})$. Substituting this into \eqref{Eqn: estimation of A_e^{(c)} for sigma and delta} gives
\[
A_{n,(a,b,c)}^{(c)}\le \sin^2(\frac{\pi c}{n})\left( 1+\frac{\cos^2(\frac{\pi c}{n})}{2} \right),
\]
which proves the $E_n^{(2)}$-case.
\end{proof}

The next lemma counts the weights $\theta_{n,(a,b,c)}$ with fixed third entry $c=j$. It will be used later to estimate $R_j$.

\begin{lemma}\label{Lem: Counting of M_j}
For $n\geq 4$, define for $1\le j\le \floor{\frac{n-1}{2}}$
\[
M_j^{(1)}\coloneqq
\sum_{(a,b,j)\in E_n^{(1)}}\theta_{n,(a,b,j)},
\qquad
M_j^{(2)}\coloneqq
\sum_{(a,b,j)\in E_n^{(2)}}\theta_{n,(a,b,j)}.
\]
For $n=2m$ even, define
\[
M_m^{(3)}\coloneqq \sum_{(a,b,m)\in E_n^{(3)}}\theta_{n,(a,b,m)}.
\]
Then
\[
M_j^{(1)}=j-1,
\qquad
M_j^{(2)}=
\begin{cases}
0, & 3j<n,\\
\frac{1}{3}, & 3j=n,\\
3j-n, & 3j>n,
\end{cases}
\]
and, in the even case,
\[
M_m^{(3)}=m-1.
\]
\end{lemma}

\begin{proof}
We first compute $M_j^{(1)}$. If $(a,b,c)\in E_n^{(1)}$ with $c=j$, then $a+b=j$ and $1\le a\le b\leq j$, so the admissible triples are exactly
\[
(a,b,j)=(a,j-a,j),\qquad 1\le a\le j-a.
\]
If $j$ is odd, there are $\frac{j-1}{2}$ such triples, all with $a<b$, hence $\theta_{n,(a,b,j)}=2$. Therefore
\[
M_j^{(1)}=2\cdot \frac{j-1}{2}=j-1.
\]
If $j$ is even, there are $\frac{j}{2}-1$ triples with $a<b$, each with $\theta_{n,(a,b,j)}=2$, and one triple $(\frac{j}{2},\frac{j}{2},j)$ with $\theta_{n,(\frac{j}{2},\frac{j}{2},j)}=1$. Hence
\[
M_j^{(1)}=2\left(\frac{j}{2}-1\right)+1=j-1.
\]

In the even case, the computation of $M_m^{(3)}$ is identical, because the triples in $E_n^{(3)}$ are precisely $(a,m-a,m)$ with $1\le a\le m-a$. Thus
\[
M_m^{(3)}=m-1.
\]

Next we compute $M_j^{(2)}$. If $(a,b,c)\in E_n^{(2)}$ and $c=j$, then
\[
1\le a\le b\le j,\qquad a+b+j=n.
\]
We consider three cases.

If $3j<n$, then $a,b,c\le j$, so
\[
a+b+c\le 3j<n,
\]
which is impossible. Hence $M_j^{(2)}=0$.

If $3j=n$, then necessarily $(a,b,c)=(j,j,j)$, and therefore
\[
M_j^{(2)}=\frac{1}{3}.
\]

Finally, suppose that $3j>n$. Then every admissible triple can be written as
\[
(a,b,c)=(n-2j+t,j-t,j),
\]
where the condition $a\le b\le j$ is equivalent to
\[
0\le t\le \left\lfloor\frac{3j-n}{2}\right\rfloor.
\]
If $3j-n$ is odd, then the first triple $(n-2j,j,j)$ has weight $\theta_{n,(n-2j,j,j)}=1$, and the remaining $\floor{\frac{3j-n}{2}}=\frac{3j-n-1}{2}$ triples have weight $\theta_{n,(n-2j+t,j-t,j)}=2$. Therefore
\[
M_j^{(2)}=1+2\cdot\frac{3j-n-1}{2}=3j-n.
\]
If $3j-n$ is even, then the first triple $(n-2j,j,j)$ and the last triple $\left(\frac{n-j}{2},\frac{n-j}{2},j\right)$ have weight $1$, while the remaining $\floor{\frac{3j-n}{2}}-1=\frac{3j-n}{2}-1$ triples have weight $\theta_{n,(n-2j+t,j-t,j)}=2$. Therefore
\[
M_j^{(2)}
=
1+2\left(\frac{3j-n}{2}-1\right)+1
=
3j-n.
\]
This completes the proof.
\end{proof}

As a second step in verifying the bound on $R_j$, define
\[
B_j\coloneqq\sum_{(a,b,j)\in E_n^{(1)}\cup E_n^{(2)}}\theta_{n,(a,b,j)} A_{n,(a,b,j)}^{(c)},
\]
for $1\le j\le \floor{\frac{n-1}{2}}$. Since $\theta_{n,(a,b,c)}^{(c)}\le \theta_{n,(a,b,c)}$, the next lemma gives an upper bound for the term in $R_j$ related to the superscript $c$.
\begin{lemma}[Upper bound for $B_j$]\label{Lem: upper bound of B_j}
For $n\ge 4$, we have
\[
B_j\leq 2(j-1)\sin^2(\frac{\pi j}{n}),\qquad 1\leq j\leq \floor{\frac{n}{2}}-1.
\]
\end{lemma}

\begin{proof}

By Lemma \ref{Lem: Upper bound of A} and Lemma \ref{Lem: Counting of M_j}, we have
\begin{equation}\label{Ineq: Estimate of B_j}
  \begin{split}
     B_j=&\sum_{(a,b,j)\in E_n^{(1)}\cup E_n^{(2)}}\theta_{n,(a,b,j)}A_{n,(a,b,j)}^{(c)}\\
\leq &\sin^2(\frac{\pi j}{n})\left[ 
\left( 1+\frac{\cos^2(\frac{\pi j}{n})}{4} \right)\sum_{(a,b,j)\in E_n^{(1)}}\theta_{n,(a,b,j)} +\left( 1+\frac{\cos^2(\frac{\pi j}{n})}{2} \right)\sum_{(a,b,j)\in E_n^{(2)}}\theta_{n,(a,b,j)} \right]\\
=&\sin^2(\frac{\pi j}{n})\left[ 
\left( 1+\frac{\cos^2(\frac{\pi j}{n})}{4} \right)M_j^{(1)} +\left( 1+\frac{\cos^2(\frac{\pi j}{n})}{2} \right)M_j^{(2)} \right].
  \end{split}
\end{equation}

If $3j<n$, then $M_j^{(2)}=0$, so \eqref{Ineq: Estimate of B_j} gives
\[
B_j\le (j-1)\sin^2(\frac{\pi j}{n})\Bigl(1+\frac{\cos^2(\frac{\pi j}{n})}{4}\Bigr)\le 2(j-1)\sin^2(\frac{\pi j}{n}).
\]

If $3j=n\geq 4$, then $\cos^2(\frac{\pi j}{n})=\cos^2(\frac{\pi}{3})=\frac{1}{4}$, $j=\frac{n}{3}\ge 2$, and $M_j^{(2)}=\frac{1}{3}$. Hence
\[
B_j
\le
\sin^2(\frac{\pi j}{n})\left[
(j-1)\left(1+\frac{1}{16}\right)
+
\frac{1}{3}\left(1+\frac{1}{8}\right)
\right]\leq 2(j-1)\sin^2(\frac{\pi j}{n}).
\]
Finally, suppose that $3j>n$. Then $M_j^{(1)}=j-1$ and $M_j^{(2)}=3j-n$, so by \eqref{Ineq: Estimate of B_j} it suffices to prove
\[
(j-1)\left( 1+\frac{\cos^2(\frac{\pi j}{n})}{4} \right)+(3j-n)\left( 1+\frac{\cos^2(\frac{\pi j}{n})}{2} \right)\leq 2(j-1).
\]
This is equivalent to
\[
(7j-2n-1)\cos^2(\frac{\pi j}{n}) \leq 4(n-2j-1).
\]
Since $7j-2n-1=j+2(3j-n)-1>j-1\ge 0$, this is in turn equivalent to
\[
\cos^2(\frac{\pi j}{n})\le \frac{4(n-2j-1)}{7j-2n-1}.
\]
Write $j=\floor{n/2}-t$. Because $j\le \floor{n/2}-1$ and $3j>n\geq 2\floor{n/2}$, we have $1\le t<\frac{\floor{n/2}}{3}$. If $n=2m+1$, then by $2t+1\le 3t$ for $t\ge 1$, the bound $t<\frac{m}{3}$, and the numerical fact $9\pi^2<128$, we have
\[
\cos^2(\frac{\pi j}{n})=\sin^2\left(\frac{(2t+1)\pi}{4m+2}\right)\le \left(\frac{(2t+1)\pi}{4m+2}\right)^2<\frac{9\pi^2 t^2}{16m^2}<\frac{8t}{3m}<\frac{8t}{3m-7t-3}=\frac{4(n-2j-1)}{7j-2n-1}.
\]
If $n=2m$, then by $t<\frac{m}{3}$ and $\pi^2<16$, we have 
\[
\cos^2(\frac{\pi j}{n})=\sin^2\left(\frac{\pi t}{2m}\right)\le\left(\frac{\pi t}{2m}\right)^2<\frac{4t}{3m}\le\frac{8t-4}{3m-7t-1}=\frac{4(n-2j-1)}{7j-2n-1}.
\]
Thus, in all cases,
\[
B_j\le 2(j-1)\sin^2\left(\frac{\pi j}{n}\right),
\qquad 1\le j\le \floor{\frac{n}{2}}-1.
\]
This proves the lemma.
\end{proof}

Now we can conclude Lemma \ref{Lem: nonnegativity of Q_2+Q_4-Q_3} in the odd case $n=2m+1\geq 7$ from the following theorem.
\begin{theorem}[Explicit choice of $\theta_{n,(a,b,c)}^{(b)}$ and $\theta_{n,(a,b,c)}^{(c)}$ for odd $n\ge 7$]
\label{Thm: Choice of theta for odd n geq 7}
For every odd $n=2m+1\ge 7$, let
\[
(a_*,b_*,c_*)=(2,m-1,m).
\]
Define
\[
\Delta_n\coloneqq B_m-2(m-1)\sin^2(\frac{\pi m}{n}),
\qquad
\delta_n\coloneqq \frac{\Delta_n}{A_{n,(a_*,b_*,c_*)}^{(c)}}.
\]
Choose
\[
\theta_{n,(a,b,c)}^{(b)}=
\begin{cases}
\delta_n, & (a,b,c)=(a_*,b_*,c_*),\\
0, & (a,b,c)\neq (a_*,b_*,c_*),
\end{cases}\qquad \theta_{n,(a,b,c)}^{(c)}=
\begin{cases}
\theta_{n,(a_*,b_*,c_*)}-\delta_n, & (a,b,c)=(a_*,b_*,c_*),\\
\theta_{n,(a,b,c)}, & (a,b,c)\neq (a_*,b_*,c_*).
\end{cases}
\]
Then $\theta_{n,(a,b,c)}^{(b)},\theta_{n,(a,b,c)}^{(c)}\ge 0$ for all $(a,b,c)\in E_n^{(1)}\cup E_n^{(2)}$. Moreover, for this choice of nonnegative numbers, the hypotheses of Proposition \ref{Prop: Decomposition of theta_e, odd n} are satisfied. Consequently,
\[
Q_{n,2}(x)+Q_{n,4}(x)-Q_{n,3}(x)\ge 0,
\qquad \text{for all }x\in\mathbb{R}_+^m.
\]
\end{theorem}

\begin{proof}
We first estimate $\Delta_n$. Recall that
\[
B_j=\sum_{(a,b,j)\in E_n^{(1)}\cup E_n^{(2)}}\theta_{n,(a,b,j)}A_{n,(a,b,j)}^{(c)},
\qquad
M_j^{(1)}=\sum_{(a,b,j)\in E_n^{(1)}}\theta_{n,(a,b,j)},
\qquad
M_j^{(2)}=\sum_{(a,b,j)\in E_n^{(2)}}\theta_{n,(a,b,j)}.
\]
Since we are in the odd case $n=2m+1\ge 7$, Lemma \ref{Lem: Counting of M_j} gives
\[
M_m^{(1)}=M_m^{(2)}=m-1.
\]
Hence
\[
 \Delta_n=B_m-(M_m^{(1)}+M_m^{(2)})\sin^2(\frac{\pi m}{n})=\sum_{(a,b,m)\in E_n^{(1)}\cup E_n^{(2)}}\theta_{n,(a,b,m)}\left( A_{n,(a,b,m)}^{(c)}-\sin^2(\frac{\pi m}{n}) \right).
\]
By \eqref{Eqn: formula of A_e^{(c)}},
\begin{equation}\label{Ineq: Lower bound of A_e^m}
A_{n,(a,b,m)}^{(c)}-\sin^2(\frac{\pi m}{n})=\cos^2(\frac{\pi m}{n})\frac{\sin^2(\frac{\pi a}{n})\sin^2(\frac{\pi b}{n})}{\sin^2(\frac{\pi a}{n})+\sin^2(\frac{\pi b}{n})}>0.
\end{equation}
Therefore,
\[
\Delta_n>0.
\]

Moreover, since $\frac{xy}{x+y}\le \frac12$ for $0<x,y<1$, we have
\[
A_{n,(a,b,m)}^{(c)}-\sin^2(\frac{\pi m}{n})\le \frac{\cos^2(\frac{\pi m}{n})}{2}.
\]
Consequently,
\[
0<\Delta_n\le \frac{\cos^2(\frac{\pi m}{n})}{2}\sum_{(a,b,m)\in E_n^{(1)}\cup E_n^{(2)}}\theta_{n,(a,b,m)}=\frac{\cos^2(\frac{\pi m}{n})}{2}(M_m^{(1)}+M_m^{(2)})=(m-1)\cos^2(\frac{\pi m}{n}) .
\]
Since $n=2m+1$,
\[
\cos^2(\frac{\pi m}{n})=\sin^2(\frac{\pi}{2n})\le \left(\frac{\pi}{2n}\right)^2.
\]
Also, the function $m\mapsto \frac{m-1}{(2m+1)^2}$ is decreasing for $m\ge 3$. Hence
\[
\Delta_n \le (m-1)\left(\frac{\pi}{2n}\right)^2 \le 2\left(\frac{\pi}{14}\right)^2 < \frac{1}{4}.
\]

To prove that $\theta_{n,(a,b,c)}^{(b)}$ and $\theta_{n,(a,b,c)}^{(c)}$ are nonnegative, it suffices to show that
\[
0\le \delta_n\le \theta_{n,(a_*,b_*,c_*)}.
\]
From the definition of $\theta_{n,(a,b,c)}$, we have $\theta_{n,(a_*,b_*,c_*)}=\theta_{n,(2,m-1,m)}\in\{1,2\}$. By \eqref{Ineq: Lower bound of A_e^m}, we have 
\[
A_{n,(a_*,b_*,c_*)}^{(c)}> \sin^2(\frac{\pi m}{n})\geq \sin^2(\frac{3\pi}{7})>\frac{1}{2}.
\]
Together with $0<\Delta_n<\frac14$, this implies
\[
0<\delta_n=\frac{\Delta_n}{A_{n,(a_*,b_*,c_*)}^{(c)}}<1\le \theta_{n,(a_*,b_*,c_*)}.
\]
Hence $\theta_{n,(a,b,c)}^{(b)},\theta_{n,(a,b,c)}^{(c)}\ge 0$ for all $(a,b,c)\in E_n^{(1)}\cup E_n^{(2)}$.\par 

We next verify the bound on $R_j$. Since only the weight $\theta_{n,(a_*,b_*,c_*)}$ for the triple $(a_*,b_*,c_*)=(2,m-1,m)$ is modified, we have
\[
  \begin{split}
    R_j&=\sum_{(a,b,j)\in E_n^{(1)}\cup E_n^{(2)}}\theta_{n,(a,b,j)}^{(c)}A_{n,(a,b,j)}^{(c)}+\sum_{(a,j,c)\in E_n^{(1)}\cup E_n^{(2)}}\theta_{n,(a,j,c)}^{(b)}A_{n,(a,j,c)}^{(b)}\\
    &=B_j-\delta_n A_{n,(a_*,b_*,c_*)}^{(c)}\mathbf 1_{\{j=m\}}+\delta_n A_{n,(a_*,b_*,c_*)}^{(b)}\mathbf 1_{\{j=m-1\}}, \qquad 1\le j\le m.
  \end{split}
\]
\noindent\textbf{Case $j=m$.}
By the definitions of $\delta_n$ and $\Delta_n$, we have
\[
R_m = B_m-\delta_n A_{n,(a_*,b_*,c_*)}^{(c)} = B_m-\Delta_n = 2(m-1)\sin^2(\frac{\pi m}{n}).
\]

\noindent\textbf{Case $1\le j\le m-2$.}
By Lemma \ref{Lem: upper bound of B_j}, we have
\[
R_j=B_j\le 2(j-1)\sin^2(\frac{\pi j}{n}).
\]

\noindent\textbf{Case $j=m-1$.}
First,
\[
\frac{A_{n,(a_*,b_*,c_*)}^{(b)}}{A_{n,(a_*,b_*,c_*)}^{(c)}}=\frac{A_{n,(2,m-1,m)}^{(b)}}{A_{n,(2,m-1,m)}^{(c)}}
=
\frac{\sin^2(\frac{2\pi}{n})+\sin^2(\frac{\pi (m-1)}{n})}{\sin^2(\frac{2\pi}{n})+\sin^2(\frac{\pi m}{n})}<1.
\]
Therefore,
\[
 R_{m-1}=B_{m-1}+\delta_n A_{n,(a_*,b_*,c_*)}^{(b)}=B_{m-1}+\frac{\Delta_n}{A_{n,(a_*,b_*,c_*)}^{(c)}} A_{n,(a_*,b_*,c_*)}^{(b)}<B_{m-1}+\Delta_n.
\]
Thus it is enough to prove that
\[
B_{m-1}+\Delta_n\le 2(m-2)\sin^2(\frac{\pi (m-1)}{n}).
\]
By \eqref{Ineq: Estimate of B_j}, we get
\[
B_{m-1}\le \sin^2(\frac{\pi (m-1)}{n})\left[(m-2)\left( 1+\frac{\cos^2(\frac{\pi (m-1)}{n})}{4} \right)+M_{m-1}^{(2)}\left( 1+\frac{\cos^2(\frac{\pi (m-1)}{n})}{2} \right) \right].
\]
Hence
\[
\begin{split}
  &2(m-2)\sin^2(\frac{\pi (m-1)}{n})-B_{m-1}\\
&\qquad \ge\sin^2(\frac{\pi (m-1)}{n})\left[2(m-2)-(m-2)\left( 1+\frac{\cos^2(\frac{\pi (m-1)}{n})}{4} \right)-M_{m-1}^{(2)}\left( 1+\frac{\cos^2(\frac{\pi (m-1)}{n})}{2} \right) \right]\\
&\qquad =\sin^2(\frac{\pi (m-1)}{n})\left[m-2-M_{m-1}^{(2)}-\frac{2M_{m-1}^{(2)}+m-2}{4}\cos^2(\frac{\pi (m-1)}{n})  \right].\\
\end{split}
\]
Since
\[
\cos^2(\frac{\pi (m-1)}{n})=\sin^2(\frac{3\pi}{2n})\le \left(\frac{3\pi}{2n}\right)^2,\quad 2M_{m-1}^{(2)}+m-2>0,\quad  \sin^2(\frac{\pi (m-1)}{n})\geq \sin^2(\frac{2\pi}{7}),
\]
we obtain
\[
2(m-2)\sin^2(\frac{\pi (m-1)}{n})-B_{m-1}> \sin^2(\frac{2\pi}{7}) \left[
m-2-M_{m-1}^{(2)} - \frac{2M_{m-1}^{(2)}+m-2}{4}
\left(\frac{3\pi}{2n}\right)^2
\right].
\]
If $n=7$, then $M_{m-1}^{(2)}=0$, so
\[
2(m-2)\sin^2(\frac{\pi(m-1)}{n})-B_{m-1}
>
\left(1-\frac{9\pi^2}{784}\right)\sin^2(\frac{2\pi}{7})
>
0.54
>
\frac14
>
\Delta_n.
\]
If $n=9$, then $M_{m-1}^{(2)}=\frac13$, so
\[
2(m-2)\sin^2(\frac{\pi(m-1)}{n})-B_{m-1}
>
\left(\frac53-\frac{\pi^2}{54}\right)\sin^2(\frac{2\pi}{7})
>
0.90
>
\frac14
>
\Delta_n.
\]
If $n=2m+1\ge 11$, then $M_{m-1}^{(2)}=m-4$, and therefore
\[
2(m-2)\sin^2\left(\frac{\pi(m-1)}{n}\right)-B_{m-1}
>
\sin^2(\frac{2\pi}{7})
\left(
2-\frac{9\pi^2(3m-10)}{16(2m+1)^2}
\right).
\]
Since for $m\geq 5$, 
\[
\frac{3m-10}{(2m+1)^2}\le \frac14,
\]
it follows that
\[
2(m-2)\sin^2\left(\frac{\pi(m-1)}{n}\right)-B_{m-1}
>
\left(2-\frac{9\pi^2}{64}\right)\sin^2(\frac{2\pi}{7})
>
0.37
>
\frac14
>
\Delta_n.
\]
Thus, in every case,
\[
2(m-2)\sin^2(\frac{\pi(m-1)}{n})-B_{m-1}>\Delta_n,
\]
and hence
\[
R_{m-1}<B_{m-1}+\Delta_n\le 2(m-2)\sin^2(\frac{\pi(m-1)}{n}).
\]
Therefore,
\[
R_j\le 2(j-1)\sin^2(\frac{\pi j}{n}),\qquad 1\le j\le m.
\]

It remains to verify the bounds on $L_{uv}$. For each fixed pair $(u,v)$ with $1\le u\le v\le m$, there is at most one triple $(a,b,c)\in E_n^{(1)}\cup E_n^{(2)}$ with $(a,b)=(u,v)$: 
\begin{itemize}
  \item In $E_n^{(1)}$, it must be $(u,v,u+v)$, provided $u+v\leq m$;
  \item In $E_n^{(2)}$, it must be $(u,v,n-u-v)$, provided $v\leq n-u-v\leq m$;
\end{itemize}
These two possibilities are mutually exclusive. Define
\[
L_{uv}^{(0)} \coloneqq \sum_{(u,v,c)\in E_n^{(1)}\cup E_n^{(2)}} \theta_{n,(u,v,c)}.
\]
Then $L_{uv}^{(0)}$ is either $0$ or the weight $\theta_{n,(a,b,c)}$ for a single triple $(a,b,c)$, and hence
\[
\begin{cases}
  L_{uv}^{(0)}\le 2
\qquad \text{ if }u<v,\\
L_{uu}^{(0)}\le 1.
\end{cases}
\]
Note that 
\[
L_{uv}
=\sum_{(u,v,c)\in E_n^{(1)}\cup E_n^{(2)}}\theta_{n,(u,v,c)}^{(c)}+\sum_{(u,b,v)\in E_n^{(1)}\cup E_n^{(2)}}\theta_{n,(u,b,v)}^{(b)}=
\begin{cases}
L_{uv}^{(0)}-\delta_n, & (u,v)=(2,m-1),\\
\delta_n, & (u,v)=(2,m),\\
L_{uv}^{(0)}, & \text{otherwise}.
\end{cases}
\]
Since $0<\delta_n<1$, we conclude that
\[
\begin{cases}
  L_{uv}\le 2
\qquad \text{ if }u<v,\\
L_{uu}\le 1.
\end{cases}
\]
Thus all hypotheses of Proposition \ref{Prop: Decomposition of theta_e, odd n} are satisfied, and the conclusion follows.
\end{proof}

\section{Proof of Lemma \ref{Lem: nonnegativity of Q_2+Q_4-Q_3} for $\mathbb{Z}_{2m}$ with $2m\geq 4$}\label{Sec: Proof of key lemma, even n}

We now treat the case $n=2m\ge 4$, following the same strategy as in Section \ref{Sec: Proof of key lemma, odd n}. Unlike the odd case, one may take the trivial decomposition $\theta_{n,(a,b,c)}^{(c)}=\theta_{n,(a,b,c)}$ and $\theta_{n,(a,b,c)}^{(b)}=0$. We therefore work directly with the original coefficients in $Q_{n,3}(x)$.

\begin{proposition}\label{Prop: Decomposition of theta_e, even n}
Let $n=2m\ge 4$. For $1\le u\le v\le m-1$, define
\[
L_{uv}\coloneqq \sum_{(u,v,c)\in E_n^{(1)}\cup E_n^{(2)}\cup E_n^{(3)}}
\theta_{n,(u,v,c)},
\]
and for $1\le j\le m$, define
\[
R_j\coloneqq \sum_{(a,b,j)\in E_n^{(1)}\cup E_n^{(2)}\cup E_n^{(3)}}
\theta_{n,(a,b,j)}A_{n,(a,b,j)}^{(c)}.
\]
Assume that
\begin{enumerate}
  \item $L_{uv}\le 2$ for all $1\le u<v\le m-1$;
  \item $L_{uu}\le 1$ for all $1\le u\le m-1$;
  \item $R_j\le 2(j-1)\sin^2(\frac{\pi j}{n})$ for all $1\le j\le m-1$;
  \item $R_m\le m-1$.
\end{enumerate}
Then
\[
Q_{n,2}(x)+Q_{n,4}(x)-Q_{n,3}(x)\ge 0,
\qquad \text{for all }x\in \mathbb{R}_+^m.
\]
\end{proposition}

\begin{proof}
We have
\[
n^3Q_{n,3}(x)=\sum_{(a,b,c)\in E_n^{(1)}\cup E_n^{(2)}\cup E_n^{(3)}}\theta_{n,(a,b,c)}w_{n,(a,b,c)}x_ax_bx_c.
\]
Fix $(a,b,c)\in E_n^{(1)}\cup E_n^{(2)}\cup E_n^{(3)}$. By the inequality $yz\le \frac{y^2}{n}+\frac{nz^2}{4}$, we obtain
\begin{equation}\label{Ineq: estimate of theta w x for c, even n}
\theta_{n,(a,b,c)}w_{n,(a,b,c)}x_ax_bx_c\le \frac{1}{n}\theta_{n,(a,b,c)}\left(\sin^2(\frac{\pi a}{n})+\sin^2(\frac{\pi b}{n})\right)x_a^2x_b^2+n\theta_{n,(a,b,c)}A_{n,(a,b,c)}^{(c)}x_c^2.
\end{equation}
Fixing $(a,b)=(u,v)$ and summing the first term on the right-hand side of \eqref{Ineq: estimate of theta w x for c, even n} over the third index $c$, we obtain
\[
\sum_{(u,v,c)\in E_n^{(1)}\cup E_n^{(2)}\cup E_n^{(3)}}\frac{1}{n}\theta_{n,(u,v,c)}\left(\sin^2(\frac{\pi u}{n})+\sin^2(\frac{\pi v}{n}) \right)x_u^2x_v^2=\frac{1}{n}L_{uv}\left(\sin^2(\frac{\pi u}{n})+\sin^2(\frac{\pi v}{n})\right)x_u^2x_v^2.
\]
Similarly, fixing $c=j$ and summing the second term on the right-hand side of \eqref{Ineq: estimate of theta w x for c, even n} over the first two indices $(a,b)$, we obtain
\[
\sum_{(a,b,j)\in E_n^{(1)}\cup E_n^{(2)}\cup E_n^{(3)}}n\theta_{n,(a,b,j)}A_{n,(a,b,j)}^{(c)}x_c^2 = nR_jx_j^2.
\]
Therefore,
\[
n^3Q_{n,3}(x)\le \frac{1}{n}\sum_{1\le u\le v\le m-1}L_{uv}\left(\sin^2(\frac{\pi u}{n})+\sin^2(\frac{\pi v}{n})\right)x_u^2x_v^2+n\sum_{j=1}^{m}R_jx_j^2.
\]
By assumptions (1) and (2),
\[
  \begin{split}
    &\sum_{1\le u\le v\le m-1}L_{uv}\left(\sin^2(\frac{\pi u}{n})+\sin^2(\frac{\pi v}{n})\right)x_u^2x_v^2\\
    &\qquad=\sum_{u=1}^{m-1} 2L_{uu}\sin^2(\frac{\pi u}{n}) x_u^4+\sum_{1\le u<v\le m-1}L_{uv}\left(\sin^2(\frac{\pi u}{n})+\sin^2(\frac{\pi v}{n})\right)x_u^2x_v^2\\
&\qquad\le\sum_{u=1}^{m-1} 2\sin^2(\frac{\pi u}{n}) x_u^4+\sum_{1\le u<v\le m-1}2\left(\sin^2(\frac{\pi u}{n})+\sin^2(\frac{\pi v}{n})\right)x_u^2x_v^2\\
&\qquad=\sum_{u=1}^{m-1}\sum_{v=1}^{m-1} \left(\sin^2(\frac{\pi u}{n})+\sin^2(\frac{\pi v}{n})\right)x_u^2x_v^2\\
&\qquad=2\left( \sum_{v=1}^{m-1} x_v^2 \right)\left( \sum_{u=1}^{m-1}\sin^2(\frac{\pi u}{n})x_u^2 \right).
  \end{split}
\]
On the other hand,
\[
\begin{aligned}
n^4Q_{n,4}(x)
&=
2\left(\sum_{k=1}^{m-1}x_k^2+\frac{1}{2}x_m^2\right)
\left(\sum_{k=1}^{m-1}\sin^2(\frac{\pi k}{n})x_k^2+\frac{1}{2}x_m^2\right)\\
&=
2\left(\sum_{k=1}^{m-1}x_k^2\right)
\left(\sum_{k=1}^{m-1}\sin^2(\frac{\pi k}{n})x_k^2\right)
+x_m^2\sum_{k=1}^{m-1}x_k^2+x_m^2\sum_{k=1}^{m-1}\sin^2(\frac{\pi k}{n})x_k^2
+\frac{1}{2}x_m^4.
\end{aligned}
\]
Hence
\[
\sum_{1\le u\le v\le m-1}L_{uv}\left(\sin^2(\frac{\pi u}{n})+\sin^2(\frac{\pi v}{n})\right)x_u^2x_v^2\le n^4Q_{n,4}(x).
\]
By assumptions (3) and (4), we have
\[
\sum_{j=1}^m R_jx_j^2\le 2\sum_{j=1}^{m-1}(j-1)\sin^2(\frac{\pi j}{n})x_j^2
+(m-1)x_m^2=n^2Q_{n,2}(x).
\]
Therefore, for all $x\in\mathbb{R}_+^m$,
\[
n^3Q_{n,3}(x)\le n^3Q_{n,4}(x)+n^3Q_{n,2}(x).
\]
This proves the proposition.
\end{proof}
Now we can conclude Lemma \ref{Lem: nonnegativity of Q_2+Q_4-Q_3} in the even case $n=2m\geq 4$ from the following theorem.
\begin{theorem}[Trivial choice of $\theta_{n,(a,b,c)}^{(b)}$ and $\theta_{n,(a,b,c)}^{(c)}$ for even $n\ge 4$]\label{Thm: Choice of theta for even n geq 4}
For every even $n=2m\ge 4$, the hypotheses of Proposition \ref{Prop: Decomposition of theta_e, even n} are satisfied. Consequently,
\[
Q_{n,2}(x)+Q_{n,4}(x)-Q_{n,3}(x)\ge 0,\qquad \text{for all }x\in \mathbb{R}_+^m.
\]
\end{theorem}

\begin{proof}
We verify the hypotheses of Proposition \ref{Prop: Decomposition of theta_e, even n}.

\noindent\textbf{Case $j=m$.}
If $(a,b,m)\in E_n^{(3)}$, then $a+b=m$. Since $n=2m$, we have
\[
\sin^2(\frac{\pi a}{n})+\sin^2(\frac{\pi b}{n})=\sin^2(\frac{\pi a}{2m})+\sin^2(\frac{\pi (m-a)}{2m})=1.
\]
Hence
\[
A_{n,(a,b,c)}^{(c)}=\frac{(\sin^2(\frac{\pi a}{n})+\sin^2(\frac{\pi b}{n})+\sin^2(\frac{\pi m}{n}))^2}{4\left(\sin^2(\frac{\pi a}{n})+\sin^2(\frac{\pi b}{n})\right)}=\frac{(1+1)^2}{4}=1.
\]
Since triples in $E_n^{(1)}\cup E_n^{(2)}$ satisfy $c\le m-1$, only triples in $E_n^{(3)}$ contribute to $R_m$. Therefore, by Lemma \ref{Lem: Counting of M_j},
\[
R_m=\sum_{(a,b,m)\in E_n^{(3)}}\theta_{n,(a,b,m)}A_{n,(a,b,m)}^{(c)}=\sum_{(a,b,m)\in E_n^{(3)}}\theta_{n,(a,b,m)}=M_m^{(3)}=m-1.
\]

\noindent\textbf{Case $1\le j\le m-1$.}
Since every triple in $E_n^{(3)}$ has $c=m$, we have
\[
R_j=\sum_{(a,b,j)\in E_n^{(1)}\cup E_n^{(2)}}\theta_{n,(a,b,j)}A_{n,(a,b,j)}^{(c)}=B_j.
\]
By Lemma \ref{Lem: upper bound of B_j}, it follows that
\[
R_j=B_j\le 2(j-1)\sin^2(\frac{\pi j}{n}),\qquad 1\le j\le m-1.
\]

It remains to verify the bounds on $L_{uv}$. For each fixed pair $(u,v)$ with $1\le u\le v\le m-1$, there is at most one triple $(a,b,c)\in E_n^{(1)}\cup E_n^{(2)}\cup E_n^{(3)}$ with $(a,b)=(u,v)$:
\begin{itemize}
  \item In $E_n^{(1)}$, it must be $(u,v,u+v)$, provided $u+v\le m-1$;
  \item In $E_n^{(3)}$, it must be $(u,v,m)$, provided $u+v=m$;
  \item In $E_n^{(2)}$, it must be $(u,v,n-u-v)$, provided $v\le n-u-v\le m-1$.
\end{itemize}
These possibilities are mutually exclusive. Hence $L_{uv}$ is either $0$ or the weight $\theta_{n,(a,b,c)}$ for a single triple $(a,b,c)$. Therefore,
\[
\begin{cases}
  L_{uv}\le 2
\qquad \text{ if }u<v,\\
L_{uu}\le 1.
\end{cases}
\]
Thus all hypotheses of Proposition \ref{Prop: Decomposition of theta_e, even n} are satisfied, and the conclusion follows.
\end{proof}

\begin{remark}
Although the direct proof above settles the even case, one can also recover it from previously known results in \cite{yao2025optimalhypercontractivitylogsobolevinequalities} and the odd case. Every even integer $n\ge 4$ can be written as $n=n_0\cdot 2^k$ with $k\ge 1$ and $n_0$ odd. The cases $n_0=1$ and $n_0=3$, that is, $n\in\{2^k,3\cdot 2^k\}$, were proved in \cite{yao2025optimalhypercontractivitylogsobolevinequalities}. For odd $n_0\ge 5$, the odd case established in Section~\ref{Sec: Proof of key lemma, odd n}, together with \cite[Theorem 4.2]{yao2025optimalhypercontractivitylogsobolevinequalities}, yields the corresponding even cases $n=n_0\cdot 2^k$. Thus all even integers $n\ge 4$ can also be covered by this indirect argument.
\end{remark}

\section*{Acknowledgments}
The author would like to thank Professor Quanhua Xu and Professor Simeng Wang for their patience and encouragement, as well as for their careful reading of the manuscript and many helpful discussions. The author is partially supported by the NSF of China (No. 12031004, No. W2441002, No. 12301161, No.12371138).

\printbibliography
\addresseshere
\end{document}